\documentclass{article}
\usepackage[all]{xy} 
\usepackage{xspace, graphicx, epsfig, amssymb, amsmath, amsthm,stmaryrd, wasysym} 
\usepackage{setspace}

\newcommand {\Z} {\mathbb{Z}}

\newcommand {\C} {\mathbb{C}}

\newcommand {\tr} {\mathrm{Tr}}

\newcommand {\St} {\mathrm{St}}

\def\N{\mathbb N}
\def\a{\mathfrak {a}}
\def\b{\mathfrak {b}}
\def\gl{\mathfrak {gl}}
\def\h{\mathfrak {h}}
\def\g{\mathfrak {g}}
\def\l{\mathfrak {l}}
\def\m{\mathfrak {m}}
\def\n{\mathfrak {n}}
\def\p{\mathfrak {p}}
\def\sl{\mathfrak{sl}}
\def\so{\mathfrak{so}}
\def\sp{\mathfrak{sp}}
\def\t{\mathfrak{t}}
\def\z{\mathfrak {z}}
\def\k{\mathfrak{k}}
\def\q{\mathfrak{q}}
\def\r{\mathfrak{r}}
\def\s{\mathfrak{s}}
\def\F{\mathfrak{F}}
\def\G{\mathfrak{G}}
\def\H{\mathfrak{H}}

\def\Ch{\mathcal{C}}
\def\Dh{\mathcal{D}}
\def\Eh{\mathcal{E}}

\newcommand {\ad}{\mathrm{ad \:}}
\newcommand {\Span}{\mathrm{Span}}

\newcommand{\Sym}{\mathrm{Sym}}

\newtheorem{thm}{Theorem}[section]
\newtheorem{lemma}[thm]{Lemma}
\newtheorem{prop}[thm]{Proposition}

\newtheorem{defn}[thm]{Definition}
\newtheorem{cor}[thm]{Corollary}

\begin{document}

\title{Parabolic and Levi subalgebras of finitary Lie algebras}
\author{Elizabeth Dan-Cohen\thanks{Research partially supported by DFG Grant PE 980/2-1 and NSF Grant DMS 0354321} \, and Ivan Penkov\thanks{Research partially supported by DFG Grant PE 980/2-1  and FAPESP Grant 2007/54820-8}}
\date{April 2, 2009}

\maketitle

\begin{abstract}
Let $\g$ be a locally reductive complex Lie algebra which admits a faithful countable-dimensional finitary representation $V$.  Such a Lie algebra is a split extension of an abelian Lie algebra by a direct sum of copies of $\sl_\infty$, $\so_\infty$, $\sp_\infty$, and finite-dimensional simple Lie algebras.  A parabolic subalgebra of $\g$ is any subalgebra which contains a maximal locally solvable (that is, Borel) subalgebra.  Building upon work by Dimitrov and the authors of the present paper, \cite{DP2}, \cite{D}, we give a general description of parabolic subalgebras of $\g$ in terms of joint stabilizers of taut couples of generalized flags.  The main differences with the Borel subalgebra case are that the description of general parabolic subalgebras has to use both the natural and conatural modules, and that the parabolic subalgebras are singled out by further ``trace conditions" in the suitable joint stabilizer.

The technique of taut couples can also be used to prove the existence of a Levi component of an arbitrary subalgebra $\k$ of $\gl_\infty$.  If $\k$ is splittable, we show that the linear nilradical admits a locally reductive complement in $\k$.  We conclude the paper with descriptions of Cartan, Borel, and parabolic subalgebras of arbitrary splittable subalgebras of $\gl_\infty$.

\vspace{5pt}
\end{abstract}

\section{Introduction}
In the present paper we study the structure of subalgebras of finitary Lie algebras, and most essentially of the three complex simple Lie algebras $\sl_\infty$, $\so_\infty$, $\sp_\infty$. We are motivated by the fundamental structural relationship between the representation theory of a locally finite Lie algebra $\g$ and the subalgebras of $\g $. Our work is a direct continuation of the papers \cite{N-P}, \cite{DP2}, \cite{D}, \cite{DPS}, as well as of the article \cite{DP4}. In these earlier papers Cartan and Borel subalgebras of $\gl_\infty$, $\sl_\infty$, $\so_\infty$, and $\sp_\infty$ were studied, but general parabolic subalgebras of $\gl_\infty$, $\sl_\infty$, $\so_\infty$, and $\sp_\infty$ were not addressed.  We fill in this gap in the present work, and we also address Levi subalgebras as well as general splittable subalgebras of $\gl_\infty$, $\sl_\infty$, $\so_\infty$, and $\sp_\infty$.

Here is a brief description of the results of the paper. Let $\g$ be one of the finitary locally finite complex Lie algebras $\gl_\infty$, $\sl_\infty$, $\so_\infty$, and $\sp_\infty$.  By $V$ we denote the natural (defining) representation of $\g$.  By $V_*$ we denote the conatural representation, i.e.\ the unique simple $\g$-submodule of the algebraic dual $V^*$ of $V$.  For $\g = \gl_\infty$ or $\sl_\infty$, the representations $V$ and $V_*$ are not isomorphic.  For $\g = \so_\infty$ or $\sp_\infty$, one has $V \cong V_*$.  Recall that a generalized flag in $V$ (or $V_*$) is a chain of subspaces characterized by two properties: see Section~\ref{tautcouples} for the definition.  In \cite{DP2} and \cite{D} generalized flags were used to describe the Borel subalgebras of $\g$.

The first key idea of the present paper is that, given any subalgebra $\k$ of $\gl_\infty$, one can attach to $\k$ a couple $\F$, $\G$ with specific properties, where $\F$ is a generalized flag in $V$ and $\G$ is a generalized flag in $V_*$, such that $\k$ is contained in the joint stabilizer of $\F$ and $\G$.  We call $\F$, $\G$ a taut couple (see Section \ref{tautcouples}).  This construction enables us to prove the existence of a Levi component of any finitary Lie algebra, i.e.\ of any subalgebra of $\gl_\infty$. We define a Levi component of a finitary Lie algebra $\g$ as a complementary subalgebra in $[\g, \g]$ of the intersection of the locally solvable radical $\r$ with $[\g, \g]$. This is a direct extension of the definition of Levi component of a finite-dimensional Lie algebra (note that our definition differs from an earlier one, compare \cite{Ba}). The existence of a Levi component is by no means obvious and is proved in Section \ref{secLevi} of this paper. We then establish some main properties of Levi components and strengthen the results for splittable subalgebras (see Section~\ref{preliminaries} for the definition of splittable).  We prove that any splittable subalgebra $\k$ has a well-defined locally reductive part $\k_{red}$ which is a complement to the linear nilradical of $\k$ (the latter is defined as the largest ideal of $\k$ consisting of nilpotent elements of $\gl_\infty$). Moreover $\k_{red}$ equals the semi-direct sum of a toral subalgebra and a Levi component of $\k$. 

Our next major result is the description of all parabolic subalgebras of $\gl_\infty$, $\sl_\infty$, $\so_\infty$, and $\sp_\infty$. Consider the case of $\gl_\infty$, and the case $\sl_\infty$ is similar, and for the cases of $\so_\infty$ and $\sp_\infty$, see Section \ref{secParsosp}. As we know from \cite{DP2}, maximal locally solvable subalgebras of $\gl_\infty$ are stabilizers of maximal closed generalized flags in the natural representation $V$ (for the definition of a closed generalized flag see Section~\ref{tautcouples} or \cite{DP2}). We show that in the above result closed generalized flags can be replaced by the more general class of semiclosed generalized flags, which we define in Section \ref{tautcouples}. Then we use the construction described earlier in this introduction and attach to any parabolic subalgebra $\p\subset\gl_\infty$ a taut couple of generalized flags $\F$, $\G$. A comparison of $\p$ with the joint stabilizer $\St_\F\cap \St_\G$ shows that $\p$ almost coincides with $\St_\F \cap \St_\G$. More precisely, the parabolic subalgebra $\p$ is singled out by ``trace conditions" on the subalgebra  $\St_\F \cap \St_\G$. This means that $\p$ and $\St_\F \cap \St_\G$ have the same linear nilradical and the same Levi components. 

There are at least two new effects produced by this result.  First, in order to describe the parabolic subalgebras of $\gl_\infty$, both representations $V$ and $V_*$ are needed instead of just one of them.  Another new effect is that not all parabolic subalgebras are self-normalizing (in fact, the self-normalizing parabolic subalgebras are precisely all joint stabilizers $\St_\F \cap \St_\G$). The most obvious example of the latter phenomenon is that $\sl_\infty$ is a parabolic subalgebra of $\gl_\infty$, as it contains a very special Borel subalgebra of $\gl_\infty$ constructed in \cite{DP2}.

We ultimately describe the parabolic subalgebras of an arbitrary splittable subalgebra $\k \subset \gl_\infty$ and show that the inclusion $\k_{red} \hookrightarrow \k$ induces a bijection of parabolic (in particular, Borel) subalgebras of $\k$ and $\k_{red}$. 

We conclude the paper with an appendix extending the existing theory of Cartan subalgebras to the case of an arbitrary splittable subalgebra of a locally reductive Lie algebra. 

\subsection*{Acknowledgements}

We thank Joseph A.\ Wolf for his supportive interest in our work.  He made many valuable suggestions, and in particular drew our attention to the work \cite{Mostow} of G.\ D.\ Mostow.

\section{Preliminaries on subalgebras of $\gl_\infty$} \label{preliminaries}

The ground field is the field of complex numbers $\C$. All vector spaces (including Lie algebras) are assumed to be at most countable dimensional.  If $\g$ is a Lie algebra, $\z(\g)$ denotes the center of $\g$.  Fix countable-dimensional vector spaces $V$ and $V_*$ and a nondegenerate pairing $\langle \cdot , \cdot \rangle \colon V \times V_*  \rightarrow \C$.  We define $\gl(V,V_*)$ (or simply $\gl_\infty$) to be the Lie algebra associated to the associative algebra $V \otimes V_*$, and we define $\sl(V,V_*)$ (or $\sl_\infty$) to be the commutator subalgebra of $\gl(V,V_*)$.  Given a symmetric nondegenerate pairing $V \times V \rightarrow \C$, we denote by $\so(V)$ (or $\so_\infty$) the Lie subalgebra $\bigwedge^2 V \subset \gl(V,V)$.  Given an antisymmetric nondegenerate pairing $V \times V \rightarrow \C$, we denote by $\sp(V)$ (or $\sp_\infty$) the Lie subalgebra $\Sym^2 (V) \subset \gl(V,V)$.

If $F$ is a subspace in $V$ or $V_*$, then $F^\perp$ stands for the orthogonal complement of $F$ (respectively in $V_*$ or $V$) with respect to the pairing $\langle \cdot , \cdot \rangle$. 

A subspace $F \subset W$, where $W$ is a vector space endowed with a symmetric or antisymmetric form, is called \emph{isotropic} if $\langle F , F \rangle = 0$.  The condition $\langle F , F \rangle = 0$ is equivalent to $F \subset F^\perp$.  A subspace $F \subset W$, where $W$ is a vector space endowed with a symmetric or antisymmetric form, is called \emph{coisotropic} if $F^\perp \subset F$.

A Lie algebra $\g$ is said to be \emph{locally finite} if every finite subset of $\g$ is contained in a finite-dimensional subalgebra.  (Clearly $\gl_\infty$, $\sl_\infty$, $\so_\infty$, and $\sp_\infty$ are locally finite.) If $\g$ is at most countable dimensional, being locally finite is equivalent to admitting an exhaustion $\g=\bigcup_{n\in\N}\g_n$ by nested finite-dimensional Lie subalgebras $\g_n$ of $\g$. If $W$ is a module over a locally finite Lie algebra $\g$, the representation is said to be \emph{finitary} if $W$ admits a basis such that all endomorphisms coming from $\g$ are given by finite matrices in this basis.  A locally finite Lie algebra $\g$ is said to be \emph{finitary} if there exists a faithful finitary representation of $\g$.  Any finitary Lie algebra is isomorphic to a subalgebra of $\gl_\infty$.

A locally finite Lie algebra is said to be \emph{locally semisimple} if it admits an exhaustion by finite-dimensional semisimple subalgebras.  

We say that a Lie algebra $\g$ is a \emph{union of reductive subalgebras} if it can be represented as a union of nested finite-dimensional reductive Lie algebras $\g_n \subset \g_{n+1}$.  A Lie algebra $\g$ is called \emph{locally reductive} if it can be expressed as a union of nested finite-dimensional reductive Lie algebras $\g_n \subset \g_{n+1}$ such that $\g_n$ is reductive in $\g_{n+1}$ (i.e.\ the induced $\g_n$-module structure on $\g_{n+1}$ is semisimple). The Lie algebras $\gl_\infty$, $\sl_\infty$, $\so_\infty$, and $\sp_\infty$ are obviously locally reductive. If $\g$ is a locally reductive Lie algebra, every element $X\in\g$ has a well-defined Jordan decomposition, and both the semisimple part $X_{ss}$ and the nilpotent part $X_{nil}$ of $X$ belong to $\g$.  (If $X \in \g_n$, then $X_{ss}$ and $X_{nil}$ are respectively the semisimple and nilpotent parts of $\ad X \colon \g_n \rightarrow \g_n$ and do not depend on $n$.)  More generally, a subalgebra $\k$ of a locally reductive Lie algebra $\g$ is said to be \emph{splittable}\footnote{We prefer ``splittable" to the term ``decomposable" used for the French ``scindable" in the English translation of N.\ Bourbaki's treatise \cite{Bourbaki}.} if for any $X \in \k$ both $X_{ss}$ and $X_{nil}$ are themselves in $\k$. 

Let $\g = \bigcup_n \g_n$ be a locally finite Lie algebra.  One says that $\g$ is \emph{locally solvable} (respectively \emph{locally nilpotent}) if every finite subset of $\g$ is contained in a solvable (resp.\ nilpotent) subalgebra.  The sum of all locally solvable ideals in $\g$ is a locally solvable ideal, so $\g$ has a unique maximal locally solvable ideal, which we call the \emph{locally solvable radical} $\r$.  The intersection $\r \cap [\g,\g]$ is a locally nilpotent ideal in $\g$, since
$$\r \cap [\g,\g] = \bigcup_n (\r \cap \g_n) \cap [\g_n , \g_n]$$
and $(\r \cap \g_n) \cap [\g_n , \g_n]$ is a nilpotent ideal of $\g_n$ for each $n$.

Let $\g $ be a finitary Lie algebra and suppose an injective homomorphism $\g \hookrightarrow \gl(V,V_*)$ is given. The \emph{linear nilradical} of $\g$ is defined as the set of elements of the locally solvable radical of $\g$ which are nilpotent as elements of $\gl(V,V_*)$.  We denote the linear nilradical of $\g$ by $\n_\g$, where the injective homomorphism $\g \hookrightarrow \gl(V,V_*)$ is understood.

\begin{prop} \label{nilcommutator}
Let $\g \hookrightarrow \gl(V,V_*)$.  Then $\n_\g$ is a locally nilpotent ideal in $\g$ such that $\n_\g \cap [\g,\g] = \r \cap [\g,\g]$.
\end{prop}

\begin{proof}
Fix $X \in \r \cap [\g,\g]$.  There exists a finite-dimensional subalgebra $\g_0 \subset \g$ such that $X \in [\g_0, \g_0]$.  Let $\r_0$ denote the solvable radical of $\g_0$, and note that $X \in \r \cap \g_0 \subset \r_0$. By finite-dimensional Lie theory, any element of $\r_0 \cap [\g_0 , \g_0]$ equals its own nilpotent part defined via $V$.  Hence $X$ is an element of $\n_\g$.  This shows that $\r \cap [\g,\g] \subset \n_\g$ and hence $\r \cap [\g,\g] = \n_\g \cap [\g,\g]$.

We now compute $[\g, \n_\g] \subset [\g , \r] \subset \r \cap [\g,\g] \subset \n_\g$.  Hence $\n_\g$ is an ideal.  Engel's theorem implies that $\n_\g$ is locally nilpotent. 
\end{proof}

\begin{lemma} \label{nilradicalcontainment}
If $\k \subset \g \hookrightarrow \gl(V,V_*)$, then  $\n_\g \cap \k \subset \n_\k$.
\end{lemma}

\begin{proof}
Note that $\n_\g \cap  \k \subset \r_\g \cap \k$, and furthermore since $\r_\g \cap \k$ is a locally solvable ideal in $\k$, one has $ \r_\g \cap \k \subset \r_\k$.  Every element of $\n_\g \cap  \k$ equals its own nilpotent Jordan component defined by the inclusion $\k \subset \gl(V,V_*)$, so as a result $\n_\g \cap \k \subset \n_\k$.
\end{proof}

\begin{lemma} \label{locreductivenilradical}
If $\g \hookrightarrow \gl(V,V_*)$ and $\g$ is a union of reductive subalgebras, then $\n_\g \subset \z(\g)$.
\end{lemma}

\begin{proof}
Fix $X \in \n_\g$ and $Y \in \g$.  There exists a reductive subalgebra $\g_0 \subset \g$ such that $X$, $Y \in \g_0$.  Since $\n_\g \cap \g_0$ is a nilpotent ideal in the reductive Lie algebra $\g_0$, one has $\n_\g \cap \g_0 \subset \z(\g_0)$.  Hence $[X , Y]=0$.
\end{proof}

The following two theorems are crucial toward the results of the present paper.

\begin{thm} \cite[Theorem 1.3]{BS} \label{irreducible}
Let $\m$ be a subalgebra of $\gl(V,V_*)$ which acts irreducibly on $V$.  Then there exists a subspace $W \subset V_*$ with $W^\perp = 0$ such that $\m$ equals $\gl(V,W)$, $\sl(V,W)$,
$\so(V)$, or $\sp(V)$, in the last two cases under an identification of $V$ and $W$ making the induced form on $V$ respectively symmetric or antisymmetric.
\end{thm}

\begin{thm} \cite{DP4} \label{locsemisimple} 
Let $\k$ be a locally semisimple subalgebra of $\gl_\infty$. 
Then $\k$ is isomorphic to a direct sum of finite-dimensional simple subalgebras and copies of $\sl_\infty$, $\so_\infty$, and $\sp_\infty$.
\end{thm}

For any locally semisimple subalgebra $\k$ of $\gl_\infty$, we introduce notation related to the decomposition of $\k$ given in Theorem~\ref{locsemisimple}.  Let $\k_0$ denote the direct sum of the finite-dimensional simple direct summands of $\k$, and let $\k_i$ denote the infinite-dimensional simple direct summands of $\k$, so that 
\begin{equation} \label{locssdecomp}
\k =  \bigoplus_{i \in I \sqcup \{0 \}} \k_i. 
\end{equation} 

The following two propositions are corollaries of Theorem~\ref{irreducible}.

\begin{prop}\label{glirreducibleeverywhere}
Suppose a subalgebra $\k \subset \bigoplus_{\gamma \in C} \gl \big(V_\gamma , (V_\gamma)_* \big)$ acts irreducibly on $V_\gamma$ and $(V_\gamma)_*$ for all $\gamma \in C$.  Then $[\k,\k]$ also acts irreducibly on $V_\gamma$ and $(V_\gamma)_*$ for all $\gamma \in C$.  
\end{prop}

\begin{proof}
Fix $\alpha \in C$, and let $\pi_\alpha \colon  \bigoplus_{\gamma \in C} \gl \big( V_\gamma , (V_\gamma)_* \big) \rightarrow  \gl \big(V_\alpha , (V_\alpha)_* \big)$ denote the projection.  We have assumed that $\k$ acts irreducibly on $V_\alpha$ and $(V_\alpha)_*$, so $\pi_\alpha (\k)$ is a subalgebra of $\gl \big( V_\alpha, (V_\alpha)_* \big)$ which acts irreducibly on both $V_\alpha$ and $(V_\alpha)_*$.  By Theorem~\ref{irreducible}, if $V_\alpha$ is infinite dimensional, then $\pi_\alpha (\k)$ is $\gl \big( V_\alpha,(V_\alpha)_* \big)$, $\sl \big( V_\alpha,(V_\alpha)_* \big)$, $\so(V_\alpha)$, or $\sp(V_\alpha)$, where in the last two cases one has a suitable identification of $V_\alpha$ and $(V_\alpha)_*$.  If $V_\alpha$ is finite dimensional, then after finite-dimensional Lie theory, $\pi_\alpha (\k) \subset \gl \big( V_\alpha , (V_\alpha)_* \big)$ is reductive.  In either case, $[\k,\k]$ acts irreducibly on $V_\alpha$ and $(V_\alpha)_*$. 
\end{proof}

\begin{prop} \label{irreducibleeverywhere}
Suppose a subalgebra $\k \subset \bigoplus_{\gamma \in C} \sl \big(V_\gamma , (V_\gamma)_* \big)$ acts irreducibly on $V_\gamma$ and $(V_\gamma)_*$ for all $\gamma \in C$.  Then

\begin{enumerate}
\item \label{uno} $\k$ is locally semisimple, and \eqref{locssdecomp} holds.

\item \label{dos} Let $C_0$ denote the set of $\gamma \in C$ for which $V_\gamma$ is finite dimensional.  Then $C \setminus C_0$ is the disjoint union of finite subsets $C_i$ for $i \in I$ such that
$$\k_i \subset \bigoplus_{\gamma \in C_i} \sl \big( V_\gamma , (V_\gamma)_* \big)$$
for $i \in I \sqcup \{0\}$.  

\item \label{tres} Let $i \in I$.  Then $\k_i$ is diagonally mapped into $\bigoplus_{\gamma \in C_i} \sl \big( V_\gamma , (V_\gamma)_* \big)$.  For $\gamma \in C_i$, the projection of $\k_i$ to $\sl \big( V_\gamma , (V_\gamma)_* \big)$ gives an isomorphism of $\k_i$ with $\sl \big(V_\gamma, (V_\gamma)_* \big)$, $\so(V_\gamma)$, or $\sp(V_\gamma)$. (In the final two cases one has an identification of $V_\gamma$ and $(V_\gamma)_*$ making the induced form on $V_\gamma$ symmetric or antisymmetric, as appropriate.)

\item \label{quatro} Each simple direct summand of $\k_0$ is contained in $\bigoplus_{\gamma \in C_0'} \sl \big( V_\gamma , (V_\gamma)_* \big)$ for some finite subset $C_0' \subset C_0$.
\end{enumerate}
\end{prop}

\begin{proof}
The projections of the proof of Proposition~\ref{irreducibleeverywhere} restrict to projections, for which we reuse the same notation, $\pi_\alpha \colon  \bigoplus_{\gamma \in C} \sl \big( V_\gamma , (V_\gamma)_* \big) \rightarrow  \sl \big(V_\alpha , (V_\alpha)_* \big)$ for $\alpha \in C$.  As in the proof of Proposition~\ref{glirreducibleeverywhere}, we see
that $\pi_\alpha (\k)$ when infinite dimensional is $\sl \big( V_\alpha,(V_\alpha)_* \big)$, $\so(V_\alpha)$, or $\sp(V_\alpha)$, where in the last two cases $\pi_\alpha (\k)$ identifies $V_\alpha$ and $(V_\alpha)_*$, making the induced form on $V_\alpha$ symmetric in the former case and antisymmetric in the latter case.
Similarly, if $V_\alpha$ is finite dimensional, then $\pi_\alpha (\k) \subset \sl \big( V_\alpha , (V_\alpha)_* \big)$ is semisimple.  

Let the finite-dimensional direct summands of the direct sum $\bigoplus_{\gamma \in C} \pi_\gamma (\k)$ be further subdivided to obtain a decomposition $\bigoplus_{j \in J} \s_j$ into simple subalgebras $\s_j$.  Thus we have $\k \subset \bigoplus_{\gamma \in C} \pi_\gamma (\k) = \bigoplus_{j \in J} \s_j$.  For any two elements $j \neq k \in J$, let $\pi_{jk} \colon \bigoplus_{l \in J} \s_l \rightarrow \s_j \oplus \s_k$ denote the projection.
For each $j \neq k \in J$, the intersection $\pi_{jk} (\k) \cap \s_j$ equals $\s_j$ if it is not trivial, since
$\pi_{jk} (\k) \cap \s_j$ is an ideal in the simple Lie algebra $\s_j$.  Note that the condition $\pi_{jk} (\k) \cap \s_j = \s_j$ is equivalent to the condition $\pi_{jk} (\k) = \s_j \oplus \s_k$.  

Now suppose $\pi_{jk} (\k) \cap \s_j = 0$.  For any $X \in \s_j$, there exists a unique $Y \in \s_k$ with $(X,Y) \in \pi_{jk} (\k)$.  This enables us to define a map $\eta_{jk} \colon  \s_j \rightarrow  \s_k$ sending $X$ to the unique element $Y \in \s_k$ with $(X,Y) \in \pi_{jk} (\k)$.  Then $\eta_{jk}$ is a Lie algebra isomorphism.

We define an equivalence relation on $J$ by setting $j \simeq k$ if $\pi_{jk} (\k) \cap \s_j = 0$.  Then (\ref{locssdecomp}) holds, where $I$ is the set of equivalence classes of $J$ for which $\s_j$ is infinite dimensional, and $\k_0$ is isomorphic to the direct sum of $\s_i$ as $i$ runs over a set of representatives of the remaining equivalence classes of $J$. 
This proves that $\k$ is locally semisimple, i.e.\ (\ref{uno}) is proved.  
 
For each element $j$ of an equivalence class $i \in I$, we have that $\s_j$ is infinite-dimensional and 
hence $\s_j = \pi_\gamma (\k)$ for some $\gamma \in C$.  
For each $i \in I$, let $C_i$ be the set of elements of $C$ corresponding in this way to the elements of the equivalence class $i$.  
Note that the sets $C_i$ are disjoint.
For each $i \in I$, $\k_i$ is the diagonal subalgebra of $\bigoplus_{j \in i} \s_j \subset \bigoplus_{\gamma \in C_i} \sl \big( V_\gamma , (V_\gamma)_* \big)$ given by the isomorphisms $\eta_{jk}$.  For $\k_i$ isomorphic to $\so_\infty$ or $\sp_\infty$ and $\gamma \in C_i$, we already observed that the projection of $\k_i$ to $\sl \big(V_\gamma , (V_\gamma)_*\big)$ yields an identification of $V_\gamma$ and $(V_\gamma)_*$. 
Thus (\ref{tres}) is proved.

Each nonzero element of $\k_i$ for $i \in I$ has nonzero components in $\sl \big( V_\gamma , (V_\gamma)_* \big)$ for all $\gamma \in C_i$, 
hence $C_i$ must be a finite set.  Because $\k$ acts irreducibly on $V_\gamma$ for all $\gamma \in C$, we conclude that $C$ is the disjoint union of the sets $C_i$ for $i \in I$ and the set $C_0$ of $\gamma \in C$ for which $V_\gamma$ is finite dimensional.  Thus (\ref{dos}) is proved, and (\ref{quatro}) comes as a result of finite-dimensional Lie theory.
\end{proof}

We conclude this section by computing the normalizers of certain diagonal subalgebras of $\gl_\infty$.

\begin{lemma}\label{infinitenormalizer}
Let $n \in \Z_{>0}$, and define $W := V \oplus \cdots \oplus V$ and $W_* := V_* \oplus \cdots \oplus V_*$ to be direct sums of $n$ copies of $V$ and $V_*$, respectively, with the natural nondegenerate pairing.  Let $\varphi$ denote the $n$-fold diagonal map
$$\varphi \colon \gl(V,V_*) \rightarrow \gl(W,W_*).$$
Then the normalizer in $\gl(W,W_*)$ of $\varphi (\sl(V,V_*))$ is $\varphi (\gl(V,V_*))$, while $\varphi (\so(V))$ and $\varphi (\sp(V))$ (defined under suitable identifications of $V$ and $V_*$) are self-normalizing in $\gl(W,W_*)$.
\end{lemma}

\begin{proof}
Suppose $X \in \gl(W,W_*)$ is in the normalizer of $\varphi (\sl(V,V_*))$.  Denote the block decomposition of $X$ by
$$X = \left( \begin{array}{cccc}
X_{11} & X_{12} & \cdots & X_{1n} \\
X_{21} & X_{22} & &\\
\vdots & & \ddots & \vdots \\
X_{n1} & & \cdots & X_{nn}
\end{array} \right) .$$
For any $A \in \sl(V,V_*)$, we have $[X,\varphi(A)] \in \varphi(\sl(V,V_*))$.  We compute that the $(i,j)$-th entry of
$$[X, \varphi(A)] = 
\left[
\left( \begin{array}{cccc}
X_{11} & X_{12} & \cdots & X_{1n} \\
X_{21} & X_{22} & &\\
\vdots & & \ddots & \vdots \\
X_{n1} & & \cdots & X_{nn}
\end{array} \right)
,
\left( \begin{array}{cccc}
A & 0 & \cdots & 0 \\
0 & A & &\\
\vdots & & \ddots & \vdots \\
0 & & \cdots & A
\end{array} \right)
\right]$$
is $[X_{ij} , A]$.  Thus for every $A \in \sl(V,V_*)$ we have $[X_{11} , A] = [X_{22}, A] = \cdots = [X_{nn} , A]$, and $[X_{ij},A] = 0$ for $i \neq j$.  Hence $X_{11} = X_{22} = \cdots = X_{nn} \in \gl(V,V_*)$, and $X_{ij} = 0$ for $i \neq j$.  This shows that the normalizer in $\gl(W,W_*)$ of $\varphi (\sl(V,V_*))$ is $\varphi (\gl(V,V_*))$.  The other cases may be proved similarly.
\end{proof}

\section{Taut couples of semiclosed generalized flags} \label{tautcouples}

We recall the notion of a generalized flag, \cite{DP1}. A \emph{chain} $\Ch$ in $V$ is any (possibly uncountable) set of nested subspaces of $V$.  That is, inclusion gives the subspaces of a chain a total ordering.  Suppose subspaces $C'$ and $C''$ in a chain $\Ch$ with $C' \subsetneq C''$ have the property that no subspaces in $\Ch$ come strictly between $C'$ and $C''$ in the inclusion ordering; then we say that $C'$ is the \emph{immediate predecessor} of $C''$, that $C''$ is the \emph{immediate successor} of $C'$, and that $C' \subset C''$ are an \emph{immediate predecessor-successor pair}.   If $\Ch$ is a chain in $V$, we denote by $\St_{\Ch,\g}$ the stabilizer of $\Ch$ in a Lie algebra $\g$ of which $V$ is a module.  If $\g$ is $\gl(V,V_*)$ or $\sl(V,V_*)$, we write simply $\St_\Ch$.

A \emph{generalized flag} is a chain $\F$ with the following two properties:
\begin{itemize}
\item[(i)] for each subspace $F \in \F$ there exists an immediate predecessor-successor pair $F' \subset F''$ with $F \in \{F' , F'' \}$;
\item[(ii)] \label{vectorpair} for each nonzero $v \in V$ there exists an immediate predecessor-successor pair $F' \subset F''$ with $v \in F''$ and $v \notin F'$.
\end{itemize}
For short, we will call any immediate predecessor-successor pair in a generalized flag $\F$ simply a \emph{pair in} $\F$. In what follows we will routinely parametrize a generalized flag $\F$ by the set $A$ of pairs in $\F$.  For any $\alpha \in A$, we denote by $F'_\alpha \subset F''_\alpha$ the pair corresponding to $\alpha$.  That is, $F'_\alpha$ is the immediate predecessor of $F''_\alpha$, and $\alpha$ is the pair $F'_\alpha \subset F''_\alpha$.  By definition, a generalized flag is exhausted by its pairs, i.e.\ $\F = \{ F'_\alpha , F''_\alpha \}_{\alpha \in A}$.  We recall that the stabilizer in $\gl(V,V_*)$ of any generalized flag $\F =  \{ F'_\alpha , F''_\alpha \}_{\alpha \in A}$ in $V$ is given by the formula $\St_\F = \sum_{\alpha \in A} F''_\alpha \otimes (F'_\alpha)^\perp$ \cite{DP2}.  

Here is a general construction from \cite{DP2} that produces a generalized flag from a chain $\Ch$ in $V$, when both $0$ and $V$ are elements of $\Ch$.  For every nonzero vector $v \in V$, let $F'(v)$ be the union of the subspaces in $\Ch$ which do not contain $v$, and let $F''(v)$ be the intersection of the subspaces in $\Ch$ which do contain $v$.  Define $\F$ to be the set $\{ F'(v) , \, F''(v) ~\colon 0 \neq v \in V \}$.  Then $\F$ is a generalized flag with the same stabilizer as $\Ch$.  

If $\Ch$ is a chain in $V$, then $\Ch^\perp := \{ C^\perp ~\colon C \in \Ch \}$ is a chain in $V_*$.  For a generalized flag $\F =  \{ F'_\alpha , F''_\alpha \}_{\alpha \in A}$ in $V$, the chain $\F^\perp = \{ (F'_\alpha)^\perp , (F''_\alpha)^\perp ~\colon \alpha \in A \}$ is not necessarily a generalized flag (for instance, it is possible to have $\F^\perp = \{0\}$).  

A subspace $F \subset V$ is \emph{closed} (in the Mackey topology) if $F=F^{\perp \perp}$.  We denote by $\overline{F}$ the closure of a subspace $F$, that is $\overline{F} := F^{\perp \perp}$.

A generalized flag $\F = \{ F'_\alpha , F''_\alpha \}_{\alpha \in A}$ in $V$ is called \emph{semiclosed} if $\overline{F'_\alpha} \in \{ F'_\alpha , F''_\alpha \}$ for every $\alpha \in A$.  In words, a generalized flag is semiclosed if the predecessor of each pair is either closed or has its successor as its closure.  A \emph{closed generalized flag} $\F = \{ F'_\alpha , F''_\alpha \}_{\alpha \in A}$ in $V$ is defined as a semiclosed generalized flag with the additional property that $F''_\alpha$ is closed for all $\alpha \in A$ \cite{DP2}.

\begin{lemma} \label{stablesubspaces}
Let $\F$ be a semiclosed generalized flag in $V$.  Any nontrivial proper closed subspace of $V$ which is stable under $\St_\F$ is both a union and an intersection of elements of $\F$.
\end{lemma}

\begin{proof}
Property (ii) of the definition of a generalized flag has the following consequence.  A proper subspace $F \subset V$ which is the union of a set of subspaces of $\F$ is also the intersection of the set of subspaces which contain $F$.  Therefore it suffices to show that a $\St_\F$-stable nonzero closed subspace $F \subset V$ is a union of elements of $\F$.

Fix $0 \neq v \in F$, and let $F'(v) \subset F''(v)$ be the unique pair in $\F$ such that $v \in F''(v)$ and $v \notin F'(v)$.  Since $\St_\F = \sum_{\alpha \in A} F''_\alpha \otimes (F'_\alpha)^\perp$, we have  
$$\St_\F \cdot v = 
\begin{cases}
F''(v) & \textrm{if } \overline{F'(v)} = F'(v) \\
F'(v) & \textrm{if }\overline{F'(v)} = F''(v).
\end{cases}$$
If $\overline{F'(v)} = F'(v)$, the $\St_\F$-stability of $\F$ yields $F''(v) \subset F$.  If $\overline{F'(v)} = F''(v)$, we have $F'(v) \subset F$, and as $F$ is closed, again $F''(v) \subset F$.  Thus $\bigcup_{0 \neq v \in F} F''(v) \subset F$.  Since we have assumed $F \neq 0$, clearly $F \subset \bigcup_{0 \neq v \in F} F''(v)$, which implies $F = \bigcup_{0 \neq v \in F} F''(v)$.
\end{proof}

\begin{lemma} \label{maximalsemiclosed}
Let $\F = \{ F'_\alpha , F''_\alpha \}_{\alpha \in A}$ be a semiclosed generalized flag in $V$.  Then $\F$ is maximal semiclosed if and only if $\dim F''_\alpha / F'_\alpha = 1$ for all $\alpha \in A$ such that $\overline{F'_\alpha} = F'_\alpha$.
\end{lemma}

\begin{proof}
Let $\tilde{\F}$ be a semiclosed generalized flag refining $\F$.  For any $\alpha \in A$ such that $\overline{F'_\alpha} = F''_\alpha$, there are no subspaces of $\tilde{\F}$ lying properly between $F'_\alpha$ and $F''_\alpha$.  Therefore if any subspace of $\tilde{\F}$ lies properly between $F'_\alpha$ and $F''_\alpha$, we have $\overline{F'_\alpha} = F'_\alpha$.  If in addition $\dim F''_\alpha / F'_\alpha = 1$ for all $\alpha$ with $\overline{F'_\alpha} = F'_\alpha$, then $\F$ is maximal semiclosed.

Conversely, assume $\F$ is maximal semiclosed.  Fix $\alpha \in A$ such that $\overline{F'_\alpha} = F'_\alpha$.  Then any subspace $F$ with $F'_\alpha \subset F \subset F''_\alpha$ and $\dim F / F'_\alpha = 1$ is closed (since it contains a closed subspace of finite codimension), hence $\tilde{\F} := \F \cup \{F\}$ is a semiclosed generalized flag refining $\F$.  As $\F$ admits no proper refinement, $\F = \tilde{\F}$, i.e.\ $\dim F''_\alpha / F'_\alpha = 1$.
\end{proof}

We say that two semiclosed generalized flags $\F$ in $V$ and $\G$ in $V_*$ form a \emph{taut couple} if the chain $\F^\perp$ is stable under $\St_\G$ and the chain $\G^\perp$ is stable under $\St_\F$.  Given a nondegenerate form $V \times V \rightarrow \C$, we call a semiclosed generalized flag $\F$ \emph{self-taut} if $\F^\perp$ is stable under the stabilizer of $\F$ in $\gl(V,V)$.

\begin{prop} \label{tautchar}
Suppose $\F$ and $\G$ are semiclosed generalized flags in $V$ and $V_*$, respectively.  The following are equivalent:
\begin{enumerate}
\item \label{tautcondition} $\F$, $\G$ form a taut couple;
\item \label{intersectioncondition} for any $F \in \F$ the subspace $F^\perp$ is both a union and an intersection of elements of $\G$, as long as $F^\perp$ is a nontrivial proper subspace of $V_*$; and vice versa (that is, for any $G \in \G$ the subspace $G^\perp$ is both a union and an intersection of elements of $\F$, as long as $G^\perp$ is a nontrivial proper subspace of $V$).
\end{enumerate}
\end{prop}

\begin{proof}
Suppose first that (\ref{intersectioncondition}) holds.  Let $F \in \F$.  If $F^\perp$ is $0$ or $V_*$, then evidently $F^\perp$ is stable under $\St_\G$.  Otherwise, $F^\perp$ is the intersection of elements stable under $\St_\G$, and thus $F^\perp$ is stable under $\St_\G$.  Similarly, the subspace $G^\perp$ is stable under $\St_\F$ for all $G \in \G$.  Hence $\F$, $\G$ form a taut couple.

Conversely, suppose $\F$, $\G$ form a taut couple.  Fix $F \in \F$.  
By the definition of a taut couple, $F^\perp$ is stable under $\St_\G$.  Observe that $F^\perp$ is closed.  If $F^\perp$ is a nontrivial proper subspace of $V_*$, then Lemma~\ref{stablesubspaces} implies that $F^\perp$ is both a union and an intersection of elements of $\G$.  The vice versa part of the statement follows immediately, by the symmetry of $\F$ and $\G$. 
\end{proof}

Let $\F = \{F'_\alpha , F''_\alpha \}_{\alpha \in A}$ and $\G= \{G'_\beta, G''_\beta \}_{\beta \in B}$ be semiclosed generalized flags in $V$ and $V_*$, respectively, and assume $\F$, $\G$ form a taut couple.  Set $C_A:=\{\alpha\in A ~\colon F_\alpha' \mathrm{~is~closed}\}$ and $C_B:=\{\beta\in B ~\colon G_\beta' \mathrm{~is~closed}\}$, and fix $\alpha \in C_A$.  Then $(F''_\alpha)^\perp \subset (F'_\alpha)^\perp$ are distinct closed subspaces of $V_*$.  Since $\F$ has no subspace properly between $F'_\alpha$ and $F''_\alpha$, Proposition~\ref{tautchar} yields that $\G$ has no closed subspace properly between $(F''_\alpha)^\perp$ and $(F'_\alpha)^\perp$.  As a result, $(F''_\alpha)^\perp$ is in $\G$ and has an immediate successor in $\G$.  That is, there exists $\beta \in C_B$ with $G'_\beta = (F''_\alpha)^\perp$.  Thus we may define a map $f_{AB} \colon C_A \rightarrow C_B$ by setting $f_{AB} (\alpha) := \beta$.  Furthermore, Proposition~\ref{tautchar} again implies $(G''_\beta)^\perp = F'_\alpha$.  Therefore $f_{AB}$ and the analogously defined map $f_{BA}$ are inverses.   This argument proves the following proposition.

\begin{prop} \label{definec}
The map $f_{AB}$ is a bijection
$$C_A = \{ \alpha \in A ~\colon F'_\alpha \textrm{ is closed} \} \rightarrow C_B = \{ \beta \in B ~\colon G'_\beta \textrm{ is closed} \}.$$
\end{prop}

Using the identification of Proposition~\ref{definec}, we denote both $C_A$ and $C_B$ by $C$.  For each $\gamma \in C$, one has $G'_\gamma = (F''_\gamma)^\perp$ and $F'_\gamma = (G''_\gamma)^\perp$.

The following proposition characterizes maximal taut couples.

\begin{prop} \label{noobstruction}
The taut couple $\F$, $\G$ is a maximal taut couple if and only if $\F$ (or $\G$) is a maximal semiclosed generalized flag.
\end{prop}

\begin{proof}
It suffices to show that, if $\F$ (or $\G$) is not maximal semiclosed, then there exists a taut couple $\tilde{\F}$, $\tilde{\G}$ such that $\tilde{\F}$ is a proper refinement of $\F$ and $\tilde{\G}$ is a proper refinement of $\G$.  In particular, this will imply that $\G$ is maximal semiclosed if and only if $\F$ is maximal semiclosed, and the statement will be proved.

If $\F$ is not maximal semiclosed, there exists $\gamma \in C$ for which $\dim {F''_\gamma} / {F'_\gamma} > 1$, by Lemma~\ref{maximalsemiclosed}.  We need only show that there exists a closed subspace $H$ with $F'_\gamma \subsetneq H \subsetneq F''_\gamma$ and $G'_\gamma \subsetneq H^\perp \subsetneq G''_\gamma$.  
Then the generalized flags $\tilde{\F} := \F \cup \{H\}$ and $\tilde{\G} := \G \cup \{H^\perp \}$ form a taut couple as desired.

Assume first that $\dim {F''_\gamma} / {F'_\gamma} < \infty$.  In this case $F''_\gamma$ and $G''_\gamma$ are both closed.  There exists a closed subspace $H$ with $F'_\gamma \subsetneq H \subsetneq F''_\gamma$, and one sees immediately that $G'_\gamma = (F''_\gamma)^\perp \subsetneq H^\perp \subsetneq (F'_\gamma)^\perp = G''_\gamma$.

Now suppose that $\dim {F''_\gamma} / {F'_\gamma} = \infty$.  The pairing $\langle \cdot , \cdot \rangle \colon V \times V_* \rightarrow \C$ yields a nondegenerate pairing ${F''_\gamma} / {F'_\gamma} \times {G''_\gamma} / {G'_\gamma} \rightarrow \C$.  Hence, by a well-known result of Mackey \cite{Mackey}, there exist dual bases in these two infinite-dimensional vector spaces.  Let $\{x_i \in F''_\gamma\}$ and $\{x_i^* \in G''_\gamma\}$ denote the preimages of such dual bases; that is, $F''_\gamma = F'_\gamma \oplus \bigoplus_i \C x_i$, and $G''_\gamma = G'_\gamma \oplus \bigoplus_i \C x_i^*$, and $\langle x_i , x_j^* \rangle = \delta_{ij}$.  One may consider the $x_i$ to be ordered by $i \in \Z$ such that two properties hold.  First, for all $i \in \Z$ and for all $v \in \overline{F''_\gamma} \setminus F''_\gamma$, there exists $N > i$ such that $\langle v , x_N^* \rangle \neq 0$.  Second, for all $i \in \Z$ and for all $w \in \overline{G''_\gamma} \setminus G''_\gamma$, there exists $M < i$ such that $\langle x_M , w \rangle \neq 0$.  This is possible because $F'_\gamma = (G''_\gamma)^\perp$ and $G'_\gamma = (F''_\gamma)^\perp$.  Let $H:= F'_\gamma \oplus \bigoplus_{i < 0} \C x_i$. Then $H^\perp = G'_\gamma \oplus \bigoplus_{i \geq 0} \C x_i^*$.  Hence $H$ is a closed subspace as required.
\end{proof}

\begin{prop} \label{splitexact}
If $\p := \St_\F \cap \St_\G \subset \gl(V,V_*)$, then the following statements hold.
\begin{enumerate}
\item \label{firstitem} $\n_\p = \sum_{\alpha \in A} F''_\alpha \otimes (F''_\alpha)^\perp$.

\item \label{seconditem} There exist vector spaces $V_\alpha$ and $(V_\beta)_*$ for $\alpha \in A$ and $\beta \in B$ such that 
\begin{itemize}
\item $F''_\alpha = F'_\alpha \oplus V_\alpha$ and $G''_\beta = G'_\beta \oplus (V_\beta)_*$;
\item $\langle V_\gamma, (V_\eta)_* \rangle = 0$ for distinct $\gamma \neq \eta \in C$;
\item $V = \bigoplus_{\alpha \in A} V_\alpha$ and $V_* = \bigoplus_{\beta \in B} (V_\beta)_*$.
\end{itemize}
Moreover,
$$\p = \n_\p \subsetplus \bigoplus_{\gamma \in C} \gl \big( V_\gamma , (V_\gamma)_* \big).$$

\item \label{thirditem} The induced isomorphism $${\p} / {\n_\p} \cong \bigoplus_{\gamma \in C} \gl \big({F''_\gamma} / {F'_\gamma} , {G''_\gamma} / {G'_\gamma} \big)$$ is independent of the choice of vector spaces  $V_\alpha$ and $(V_\beta)_*$.

\end{enumerate}
\end{prop}

\begin{proof} 
We show first that $\p = \sum_{\gamma \in C} F''_\gamma \otimes G''_\gamma + \sum_{\alpha \in A \setminus C} F''_\alpha \otimes (F''_\alpha)^\perp$.  

Note that $ \sum_{\gamma \in C} F''_\gamma \otimes G''_\gamma + \sum_{\alpha \in A \setminus C} F''_\alpha \otimes (F''_\alpha)^\perp$ stabilizes both $\F$ and $\G$.  
Clearly both terms $\sum_{\gamma \in C} F''_\gamma \otimes G''_\gamma$ and $\sum_{\alpha \in A \setminus C} F''_\alpha \otimes (F''_\alpha)^\perp$ stabilize $\F$, while the term $\sum_{\gamma \in C} F''_\gamma \otimes G''_\gamma$ stabilizes $\G$.  To show that $F''_\alpha \otimes (F''_\alpha)^\perp$ stabilizes $G''_\beta$ for all $\alpha \in A$ and $\beta \in B$, we observe that if $(F''_\alpha \otimes (F''_\alpha)^\perp) \cdot G''_\beta \neq 0$ then $\langle F''_\alpha , G''_\beta \rangle \neq 0$.  Hence $(F''_\alpha \otimes (F''_\alpha)^\perp) \cdot G''_\beta \subset (F''_\alpha)^\perp \subset G''_\beta$.

Consider now $X \in \St_\F \cap \St_\G$.  Since $X \in \St_\F$, we may express $X = \sum_{i=1}^n v_i \otimes w_i$, with $v_i \in F''_{\alpha_i} \setminus F'_{\alpha_i}$ and $w_i \in (F'_{\alpha_i})^\perp$.  Recall that $(F''_\gamma)^\perp = G'_\gamma \subset G''_\gamma \subset (F'_\gamma)^\perp$ for each $\gamma \in C$.  We assume that for each $\gamma \in C$, the vectors $w_i$ such that $w_i \in (F'_\gamma)^\perp \setminus G''_\gamma$ are linearly independent modulo $G''_\gamma$.  

Assume, for the sake of a contradiction, that the set of $\gamma \in C$ for which there exists $i \in \{1, \ldots , n\}$ such that $w_i \in (F'_\gamma)^\perp \setminus G''_\gamma$ is nonempty.  Let $\gamma$ denote the maximal element of that set.  Let $I$ denote the set of $i \in \{1 , \ldots , n \}$ such that $w_i \in (F'_\gamma)^\perp \setminus G''_\gamma$.  Fix $y \in G''_\gamma$, and compute
\begin{eqnarray*}
X \cdot y & = & (\sum_{i=1}^n v_i \otimes w_i) \cdot y 
 = - \sum_{i=1}^n \langle v_i , y \rangle w_i \\
& = & - \sum_{i \in I} \langle v_i , y \rangle w_i - \sum_{i \notin I} \langle v_i , y \rangle w_i .
\end{eqnarray*}
 If $w_i \notin (F'_\gamma)^\perp$, then $\langle v_i , y \rangle = 0$.  Hence the term $\sum_{i \notin I} \langle v_i , y \rangle w_i$
 is in $G''_\gamma$, since for each $i \notin I$ either $w_i \in G''_\gamma$ or $w_i \notin (F'_\gamma)^\perp$. 
 As $X \cdot y \in G''_\gamma$, it follows that the term $\sum_{i \in I} \langle v_i , y \rangle w_i$ is also in $G''_\gamma$.  By hypothesis the vectors $w_i$ for $i \in I$ are linearly independent modulo $G''_\gamma$, and thus $\langle v_i , y \rangle = 0$ for $i \in I$.  Since $y$ is arbitrary, we have shown that $v_i \in (G''_\gamma)^\perp = F'_\gamma$.  This contradicts the assumption that $v_i \in F''_\gamma \setminus F'_\gamma$. Therefore there are no $\gamma \in C$ and $i \in \{1, \ldots , n\}$ such that $w_i \in (F'_\gamma)^\perp \setminus G''_\gamma$, and we have shown $\p = \sum_{\gamma \in C} F''_\gamma \otimes G''_\gamma + \sum_{\alpha \in A \setminus C} F''_\alpha \otimes (F''_\alpha)^\perp$.

The linear nilradical of $\St_\F$ is $\sum_{\alpha \in A} F''_\alpha \otimes (F''_\alpha)^\perp$ \cite{DP2}, and it is contained in $\p$ because $(F''_\gamma)^\perp = G'_\gamma \subset G''_\gamma$ for all $\gamma \in C$.  Lemma \ref{nilradicalcontainment} enables us to conclude that  $\sum_{\alpha \in A} F''_\alpha \otimes (F''_\alpha)^\perp$ is a locally nilpotent ideal contained in the linear nilradical of $\p$.  
 
We now show the existence of subspaces $V_\alpha$ and $(V_\beta)_*$ as in statement (\ref{seconditem}).  Let $\tilde{\F} = \{ \tilde{F}'_\alpha , \tilde{F}''_\alpha\}_{\alpha \in \tilde{A}}$, $\tilde{\G} = \{ \tilde{G}'_\beta , \tilde{G}''_\beta\}_{\beta \in \tilde{B}}$ be a maximal refinement of the taut couple $\F$, $\G$.  By Proposition~\ref{noobstruction}, both $\tilde{\F}$ and $\tilde{\G}$ are maximal semiclosed generalized flags.  Moreover, $\tilde{\F}$ and $\tilde{\G}$ are both maximal closed generalized flags.  This fact is seen in Theorem~\ref{tfae}, and we use it now for convenience, as there is no logical obstruction.
  Let $\tilde{C}$ denote the set analogous to $C$ as defined in Proposition~\ref{definec}.  It is shown in \cite{DP2} that there exist bases of $V$ and $V_*$ compatible with the maximal closed generalized flags $\tilde{\F}$ and $\tilde{\G}$ in the following sense.  For each $\alpha \in \tilde{A}$, the set of basis vectors in $\tilde{F}''_\alpha \setminus \tilde{F'}_\alpha$ is a basis for the quotient ${\tilde{F}''_\alpha} / {\tilde{F'}_\alpha}$, and similarly for each $\beta \in \tilde{B}$.  Note that $\dim {\tilde{F}''_\gamma} / {\tilde{F'}_\gamma} = \dim  {\tilde{G}''_\gamma} / {\tilde{G'}_\gamma} = 1$ for all $\gamma \in \tilde{C}$.  Let $v_\gamma \in V$ and $v^\gamma \in V_*$ denote the basis vectors corresponding to the pair $\gamma \in \tilde{C}$.  One may assume, according to \cite{DP2}, that these vectors are dual in the sense that $\langle v_\gamma , v^\eta \rangle = \delta_{\gamma \eta}$ for all $\gamma$, $\eta \in \tilde{C}$.  For each $\alpha \in A$, take $V_\alpha$ to be the span of all the basis elements of $V$ in $F''_\alpha \setminus F'_\alpha$.  For each $\beta \in B$, take $(V_\beta)_*$ to be the span of all the basis elements of $V_*$ in $G''_\beta \setminus G'_\beta$.  Then these vector subspaces have the desired properties.

Let $V_\alpha$ and $(V_\beta)_*$ be as in statement (\ref{seconditem}).  To check that for each $\gamma \in C$ the restriction of the pairing to $V_\gamma \times (V_\gamma)_*$ is nondegenerate, suppose $0 \neq v \in V_\gamma$.  Then since $V \notin F'_\gamma$, there exists $w \in G''_\gamma$ such that $\langle v , w \rangle \neq 0$.  There exist elements $w_1 \in G'_\gamma$ and $w_2 \in (V_\gamma)_*$ such that $w = w_1 + w_2$, and hence $0 \neq \langle v , w \rangle = \langle v , w_1 + w_2  \rangle = \langle v , w_2 \rangle$.  Similarly, any nontrivial element of $(V_\gamma)_*$ pairs nontrivially with some element of $V_\gamma$.  Furthermore, we have assumed that $\langle V_\gamma , (V_\eta)_* \rangle = 0$ for $\gamma \neq \eta \in C$.  We have 
\begin{eqnarray*}
\sum_{\gamma \in C} F''_\gamma \otimes G''_\gamma &= &  \sum_{\gamma \in C} F''_\gamma \otimes \big((F''_\gamma)^\perp \oplus (V_\gamma)_* \big)  \\
& = & \sum_{\gamma \in C}  F''_\gamma \otimes (F''_\gamma)^\perp \oplus F'_\gamma \otimes (V_\gamma)_* \oplus V_\gamma \otimes (V_\gamma)_*.
\end{eqnarray*}
Observe that $\sum_{\gamma \in C} F'_\gamma \otimes (V_\gamma)_* \subset \sum_{\alpha \in A} F''_\alpha \otimes (F''_\alpha)^\perp$, since for any $v \in F'_\gamma$, one has $v \in F''_\beta$ for some $\beta < \gamma$, so $(V_\gamma)_* \subset (F'_\gamma)^\perp \subset (F''_\beta)^\perp$.  Thus we have established a vector space decomposition $\p = (\bigoplus_{\gamma \in C} V_\gamma \otimes (V_\gamma)_*) \oplus (\sum_{\alpha \in A} F''_\alpha \otimes (F''_\alpha)^\perp)$.  The quotient ${\p} / {(\sum_{\alpha \in A} F''_\alpha \otimes (F''_\alpha)^\perp)}$ is isomorphic to $\bigoplus_{\gamma \in C} V_\gamma \otimes (V_\gamma)_* = \bigoplus_{\gamma \in C} \gl \big( V_\gamma, (V_\gamma)_* \big)$.  Since this quotient is locally reductive, every locally nilpotent ideal is central.
Hence any element of $\n_\p$ maps to an element in the span of central elements in copies of finite-dimensional general linear Lie algebras, but the nilpotence of elements of $\n_\p$ implies that their image must be trivial.  
 It follows that the linear nilradical is $\n_\p = \sum_{\alpha \in A} F''_\alpha \otimes (F''_\alpha)^\perp$, which is (\ref{firstitem}).  We have also shown (\ref{seconditem}).

From the definition of $V_\gamma$ and $(V_\gamma)_*$, we see that  ${\p} / {\n_\p} \cong \bigoplus_{\gamma \in C} \gl \big({F''_\gamma} / {F'_\gamma} , {G''_\gamma} / {G'_\gamma} \big)$.  To prove (\ref{thirditem}), we must show that this isomorphism does not depend of the choice of vector space complements.  We will show that the actions of $\n_\p$ on ${F''_\gamma} / {F'_\gamma}$ and ${G''_\gamma} / {G'_\gamma}$ for any $\gamma \in C$ are trivial.  We have $\n_\p = \sum_{\alpha \in A} F''_\alpha \otimes (F''_\alpha)^\perp$.  Fix $\gamma \in C$.  If $\alpha \geq \gamma$, then $(F''_\alpha \otimes (F''_\alpha)^\perp) \cdot F''_\gamma = \langle F''_\gamma, (F''_\alpha)^\perp \rangle F''_\alpha = 0$.  If $\alpha < \gamma$, then $(F''_\alpha \otimes (F''_\alpha)^\perp) \cdot F''_\gamma = \langle F''_\gamma, (F''_\alpha)^\perp \rangle F''_\alpha \subset F''_\alpha \subset F'_\gamma$.  Thus in either case $\n_\p \cdot F''_\gamma \subset F'_\gamma$, i.e.\ $\n_\p \cdot {F''_\gamma} / {F'_\gamma} = 0$.  Similarly, $\n_\p$ acts trivially on ${G''_\gamma} / {G'_\gamma}$.  It follows that the isomorphism ${\p} / {\n_\p} \cong \bigoplus_{\gamma \in C} \gl \big({F''_\gamma} / {F'_\gamma} , {G''_\gamma} / {G'_\gamma} \big)$ does not depend on the choice of vector space complements.
\end{proof}

The following is a key construction.  It places an arbitrary subalgebra of $\gl_\infty$ tightly within the stabilizer of a taut couple.

\begin{thm} \label{tautcouple}
Let $\k \subset \gl(V,V_*)$ be any subalgebra. There exists a $\k$-stable taut couple $\F$, $\G$ such that ${F''_\gamma} / {F'_\gamma}$ and ${G''_\gamma} / {G'_\gamma}$ are irreducible $\k$-modules for all $\gamma \in C$.  Furthermore, one has $\n_\k = \n_{\St_\F \cap \St_\G} \cap \k$.
\end{thm}

\begin{proof}
Let $\Ch$ be a maximal chain of \emph{closed} $\k$-stable subspaces of $V$, and let $\Dh$ be a maximal chain of $\k$-stable subspaces of $V$ containing $\Ch$. For any $F \subset V$ we define $$P(F) := \bigcap_{\overline{D} = F, \, D \in \Dh} D.$$
The subspace $P(F)$ is an intersection of subspaces in the chain, and hence $P(F)$ is stable under $\k$ and $\Dh \cup \{ P(F) \}$ is a chain.  By the maximality of $\Dh$, $P(F) \in \Dh$ for any $F \subset V$. 
 
We claim that $\overline{P(\overline{D})} = \overline{D}$ for any $D \in \Dh$.  To check this, assume for the sake of a contradiction that $\overline{P(\overline{D})} \neq \overline{D}$.  Then $P(\overline{D})$ is closed, since the maximality of $\Ch$ implies that the closed subspace $\overline{P(\overline{D})}$, which forms a chain together with $\Ch$ and is stable under $\k$, is in the chain $\Ch$.  
Since $P(\overline{D})$ is defined as the intersection of subspaces whose closure is $\overline{D}$, there must exist a non-closed subspace properly between $P(\overline{D})$ and $\overline{D}$, and hence $\dim \overline{D} / P(\overline{D}) = \infty$.
Then $\k$ stabilizes the generalized flag
$$ 0 \subset P(\overline{D}) \subset \H \subset \overline{D} \subset V,$$
where $\H$ is produced from the chain  $\{ H \in \Dh ~\colon \overline{H} = \overline{D} \}$ according to the general procedure that produces a generalized flag with the same stabilizer as a given chain, as described at the beginning of Section~\ref{tautcouples}.  Let $E$ denote the set of pairs in $\H$, so that $\H = \{H'_\epsilon , H''_\epsilon \}_{\epsilon \in E}$.  By the definition of $P(\overline{D})$, we know that $\overline{H'_\epsilon} = \overline{D}$, and hence $(H'_\epsilon)^\perp = D^\perp$ for all $\epsilon \in E$.
Therefore
\begin{eqnarray*}
\k & \subset & P(\overline{D}) \otimes V_* + \sum_\epsilon H''_\epsilon \otimes (H'_\epsilon)^\perp +  V \otimes D^\perp \\
& = & P(\overline{D}) \otimes V_* + \sum_\epsilon H''_\epsilon \otimes D^\perp +  V \otimes D^\perp \\
& = & P(\overline{D}) \otimes V_* +  V \otimes D^\perp.
\end{eqnarray*}
As a result, $\k$ stabilizes every subspace between $P(\overline{D})$ and $\overline{D}$.  As there is an abundance of closed subspaces between $P(\overline{D})$ and $\overline{D}$ (for instance, the subspace $P(\overline{D}) \oplus \C x$ for any $x \in \overline{D} \setminus P(\overline{D})$), this contradicts the hypothesis that $\Ch$ is maximal with respect to closed $\k$-stable subspaces.  Thus we have shown $\overline{P(\overline{D})} = \overline{D}$ for all $D \in \Dh$.

Consider the chain $\Eh := \{P(\overline{D}), \, \overline D  ~\colon  D \in \Dh\}$.  Note that $0$ and $V$ are elements of $\Eh$, since both are closed and $\k$-stable.  The general construction described at the beginning of Section~\ref{tautcouples} produces a generalized flag $\F$ with the same stabilizer as $\Eh$.

We have now constructed a generalized flag $\F$ with $\k \subset \St_\F$.  The next step is to show that $\F$ is a semiclosed generalized flag.  Suppose $F' \subset F''$ is a pair in $\F$ and $F'$ is not closed.  By the maximality of $\Dh$, $\F$ is a subchain of $\Dh$.  So $F' \in \Dh$, and 
$\overline{F'}$ is a closed $\k$-stable subspace such that $\Ch \cup \{ \overline{F'} \}$ is a chain.  By the maximality of $\Ch$, one has $\overline{F'} \in \Ch \subset \Dh$.  Since $P(F') \subset F'$ and there are no subspaces in the chain $\{ P(\overline{D}), \overline D ~\colon D \in \Dh \}$ properly between $P(\overline{F'})$ and $\overline{F'}$, we see that it must be the case that $F' = P(F')$.  For any $D \in \Dh$ with $F' \subset D \subset \overline{F'}$, one has $\overline{D} = \overline{F'}$, and hence $P(\overline{D}) = P(\overline{F'})$.  This shows $F'' = \overline{F'}$.

Note also that by construction the quotient $F'' / F'$ is an irreducible $\k$-module.  (When $\alpha \in A \setminus C$, ${F''_\alpha} / {F'_\alpha} $ is a trivial module.)

To obtain a semiclosed generalized flag $\G$ in $V_*$ with the desired properties, repeat the above construction, starting with the chain of closed $\k$-stable subspaces $\Ch^\perp$.  That is, take a maximal chain $\Dh$ of $\k$-stable subspaces of $V_*$, containing the maximal chain of closed $\k$-stable subspaces $\Ch^\perp$.  

We now demonstrate that $\F$, $\G$ form a taut couple.  Note that the closures of the subspaces appearing in $\F$ were elements of the chain $\Ch$, by the maximality of $\Ch$.  That is, $\F^{\perp \perp} \subset \Ch$, which implies $\F^{\perp} \subset \Ch^\perp$.  Since we used the chain $\Ch^\perp$ to construct the generalized flag $\G$, the chain $\F^\perp$ is stable under $\St_\G$.  Furthermore, the closure of any subspace of $\G$ forms a chain together with $\Ch^\perp$, and hence $\G^\perp \cup \Ch^{\perp \perp} = \G^\perp \cup \Ch$ is a chain.  The maximality of $\Ch$ implies $\G^\perp \subset \Ch$, and thus since the chain $\Ch$ was used in the construction of $\F$, we see that $\G^\perp$ is stable under $\St_\F$.

It remains to show that $\n_\k = \n_{\St_\F \cap \St_\G} \cap \k$.
Set $\p := \St_\F \cap \St_\G$.   Lemma~\ref{nilradicalcontainment} yields $\n_\p \cap \k \subset \n_\k$.  Consider the homomorphism $\k \rightarrow {\p} / {\n_\p}$ with kernel $\n_ \p \cap \k$.  Recall from Proposition~\ref{splitexact} that ${\p} / {\n_\p} = \bigoplus_{\gamma \in C} \gl\big({F''_\gamma} / {F'_\gamma} , {G''_\gamma} / {G'_\gamma}\big)$.  The image of $\k$ in ${\p} / {\n_\p}$ acts irreducibly on ${F''_\gamma} / {F'_\gamma}$ and on ${G''_\gamma} / {G'_\gamma}$ for all $\gamma \in C$.  

Consider the composition $\k \rightarrow {\p} / {\n_\p} \rightarrow \gl\big({F''_\gamma} / {F'_\gamma} , {G''_\gamma} / {G'_\gamma}\big)$ for a fixed $\gamma \in C$.  The image of $\k$ under this homomorphism  is a subalgebra of $ \gl\big({F''_\gamma} / {F'_\gamma} , {G''_\gamma} / {G'_\gamma}\big)$ which acts irreducibly on both ${F''_\gamma} / {F'_\gamma}$ and ${G''_\gamma} / {G'_\gamma}$.  
Suppose first that $\dim {F''_\gamma} / {F'_\gamma} = \infty$.  Then Theorem~\ref{irreducible} applies, from which we conclude that the image of $\k$ in $\gl\big({F''_\gamma} / {F'_\gamma} , {G''_\gamma} / {G'_\gamma}\big)$ is one of $\gl\big({F''_\gamma} / {F'_\gamma} , {G''_\gamma} / {G'_\gamma}\big)$, $\sl \big({F''_\gamma} / {F'_\gamma} , {G''_\gamma} / {G'_\gamma} \big)$, $\so({F''_\gamma} / {F'_\gamma})$, and $\sp({F''_\gamma} / {F'_\gamma})$.  None of these subalgebras has any nontrivial locally nilpotent ideals; hence the image of $\n_\k$ in $\gl\big({F''_\gamma} / {F'_\gamma} , {G''_\gamma} / {G'_\gamma}\big)$ is trivial.  Now suppose $\dim {F''_\gamma} / {F'_\gamma} < \infty$.  Then the image of $\k$ in $\gl\big({F''_\gamma} / {F'_\gamma} , {G''_\gamma} / {G'_\gamma}\big)$  acts irreducibly on ${F''_\gamma} / {F'_\gamma}$ and ${G''_\gamma} / {G'_\gamma}$, and hence any locally nilpotent ideal in the image of $\k$ is central in $\gl\big({F''_\gamma} / {F'_\gamma} , {G''_\gamma} / {G'_\gamma}\big)$.  

It follows that any element of $\n_\k$ maps to an element in the span of central elements in copies of finite-dimensional general linear Lie algebras, but the Jordan nilpotence of elements of $\n_\k$ implies that their image must be trivial.
Therefore the image of $\n_\k$ in ${\p} / {\n_\p}$ is trivial, i.e $\n_\k \subset \n_\p \cap \k$.  Thus we have shown $\n_\k = \n_\p \cap \k$.
\end{proof}

\begin{prop} \label{uniquetautpair}
The map from taut couples of generalized flags in $V$ and $V_*$ to subalgebras of $\gl(V,V_*)$ given by $$(\F,\G) \mapsto \St_\F \cap \St_\G$$ is injective. 
\end{prop}

\begin{proof}
It suffices to show that one may reconstruct the subchain $\{F''_\alpha ~\colon \alpha \in A \}$ from $\St_\F \cap \St_\G$.  For any $0 \neq v \in F$, one may compute, using the formula for the stabilizer of a taut couple given in Proposition~\ref{splitexact},
$$(\St_\F \cap \St_\G) \cdot v = 
\begin{cases}
F'' & \textrm{if } \overline{F'} = F' \\
F' & \textrm{if }\overline{F'} = F'',
\end{cases}$$
where $F' \subset F''$ is the pair given by the definition of a generalized flag such that $v \in F''$ and $v \notin F'$. 
The set of immediate successors in $\F$ is therefore obtained by taking for every nonzero $v \in V$ the subspace 
$(\St_\F \cap \St_\G) \cdot v$ if  $v \in (\St_\F \cap \St_\G) \cdot v$, or the subspace $$\overline{(\St_\F \cap \St_\G) \cdot v}$$ if $v \notin (\St_\F \cap \St_\G) \cdot v$.
\end{proof}

\section{Levi components and locally reductive parts} \label{secLevi}

\begin{defn} \label{levicomponent}
Let $\g$ be a locally finite Lie algebra, and let $\r$ denote its locally solvable radical.  We say that a subalgebra $\l$ is a \emph{Levi component} of $\g$ if $$[\g,\g] = (\r \cap [\g,\g]) \subsetplus \l.$$
\end{defn}

We first prove the existence of Levi components of finitary Lie algebras.
\begin{thm} \label{existence}
Any finitary Lie algebra has a Levi component.
\end{thm}

\begin{proof}
Let $\g$ be a finitary Lie algebra.  
Then $\g$ has a finitary representation $V$, and one may consider $\g$ as a subalgebra of $\gl(V,V_*)$.  Let $\F$, $\G$ be a taut couple as given in Theorem~\ref{tautcouple}.  By Theorem~\ref{tautcouple}, $\g \subset \p$ and $\n_\g = \n_\p \cap \g$, where $\p := \St_\F \cap \St_\G$.  Let $\pi \colon \g \rightarrow {\p} / {\n_\p}$ be the inclusion of $\g$ into $\p$ followed by the quotient map.  
Recall from Proposition~\ref{splitexact} that $$\p /  \n_\p = \bigoplus_{\gamma \in C}\gl\big({F''_\gamma} / {F'_\gamma} , {G''_\gamma} / {G'_\gamma}\big).$$  

Consider $\m := \pi([\g,\g]) \subset \bigoplus_{\gamma \in C} \sl\big({F''_\gamma} / {F'_\gamma} , {G''_\gamma} / {G'_\gamma}\big)$.  By Proposition~\ref{glirreducibleeverywhere}, $\m$ acts irreducibly on ${F''_\gamma} / {F'_\gamma}$ and ${G''_\gamma} / {G'_\gamma}$ for all $\gamma \in C$ because $\g$ does.  By Proposition~\ref{irreducibleeverywhere}, $\m$ is locally semisimple.   

We obtain a pullback $\l \subset \g$ of $\m$ as follows.  Let $\m = \bigcup_i \m_i$ be an exhaustion by finite-dimensional semisimple subalgebras $\m_i$.  There exist nested finite-dimensional subalgebras $\k_i \subset \pi^{-1} (\m_i)$ such that $\pi(\k_i) = \m_i$.  We can choose inductively nested Levi components $\l_i$ of $\k_i$, because any maximal semisimple subalgebra is a Levi component.  Then $\pi|_{\l_i}$ is an isomorphism of $\l_i$ with $\m_i$, and $\pi$ gives an isomorphism of $\l := \bigcup_i \l_i$ with $\m$.

We will show that $\l$ is a Levi component of $\g$.  We know $\n_\g \subsetplus \l$ is an ideal in $\g$ since it is the pullback of an ideal in $\pi(\g) = {\g} / (\n_\p \cap \g) = {\g} / {\n_\g}$.  Moreover, the quotient ${\g} / {(\n_\g \subsetplus \l)}$ is abelian since the corresponding quotient of ${\g} / {\n_\g}$ is abelian.  Thus $[\g,\g] \subset \n_\g \subsetplus \l$.  Since $\l$ is locally semisimple, we have $\l \subset [\g,\g]$.  So $\l \subset [\g,\g] \subset \n_\g \subsetplus \l$, which implies $[\g,\g] = (\n_\g \cap [\g,\g]) \subsetplus \l$.  But we know from Proposition~\ref{nilcommutator} that $\n_\g \cap [\g,\g] = \r \cap [\g,\g]$, so $\l$ is a Levi component of $\g$.
\end{proof}

For any subalgebra $\g \subset \gl_\infty$, we have the following chain of ideals in $\g$:
$$ [\g,\g] \, \subset \, \n_\g + [\g,\g] \, \subset \, \r + [\g,\g] \, \subset \, \g.$$
For any Levi component $\l$ of $\g$, one has $\n_\g + [\g,\g] = \n_\g \subsetplus \l$ and $\r + [\g,\g] = \r \subsetplus \l$.  The first two terms of the filtration $[\g,\g]$ and $\n_\g + [\g,\g]$ are splittable subalgebras of $\gl_\infty$.  Note that, unlike in the case of finite-dimensional $\g$, the ideal $\r \subsetplus \l$ can be strictly contained in $\g$.

It follows from the proof of Theorem \ref{existence} that any finitary Lie algebra $\g$ admits an exhaustion $\bigcup_i \g_i$ by finite-dimensional Lie algebras $\g_i$ so that the union $\bigcup_i \l_i$ of certain Levi components $\l_i$ of $\g_i$ is a Levi component of $\g$.  It is not true, however, (and this is why our definition of Levi component is more restrictive than the one given in \cite{Ba}) that for any exhaustion $\g = \bigcup_i \g_i$ and for any choice of nested Levi components $\l_i \subset \l_{i+1}$, the union $\bigcup_i \l_i$ is a Levi component of $\g$. 
Indeed, recall the following example from \cite[Example 2]{DP4}.  Consider dual bases $\{v_i ~\colon i \in \Z_{>0} \}$ of $V$ and $\{v_i^* ~\colon i \in \Z_{>0} \}$ of $V_*$.  Let $\tilde{V} := V \oplus \C \tilde{v}$, and  define $\langle \tilde{v} , v_i^* \rangle := 1$ for all $i \in \Z_{>0}$.  Then $V_*$ pairs nondegenerately with both $V$ and $\tilde{V}$.  One has $\g := \sl(V,V_*) \cong \sl_\infty$ properly contained in $\tilde{\g} := \sl(\tilde{V}, V_*) \cong \sl_\infty$.  Consider $\g_n := \sl ( \Span \{ v_1 , v_2 , \ldots v_n \} , \Span \{ v_1^* , v_2^* , \ldots v_n^* \} )$ and $\tilde{\g}_n := \tilde{\g} \cap ( \Span \{ \tilde{v}  , v_1 , v_2 , \ldots v_n \} \otimes \Span \{ v_1^* , v_2^* , \ldots v_n^* \} )$.  Then $\g_n$ is a Levi component of $\tilde{\g}_n$.
Nevertheless the union of Levi components $\bigcup_n \g_n = \g$ is not a Levi component of $\tilde{\g} = \bigcup_n \tilde{\g}_n$.

We now demonstrate a few properties of Levi components.

\begin{thm} \label{maxlocss}
Let $\g$ be a subalgebra of $\gl(V,V_*)$.  Then any Levi component of $\g$ is a maximal locally semisimple subalgebra of $\g$.  Furthermore, any two Levi components of $\g$ are isomorphic.
\end{thm}

\begin{proof} 
Let $\l$ be any Levi component of $\g$, and let $\l_0$ denote the locally semisimple Levi component of $\g$ constructed in the proof of Theorem~\ref{existence}.  Since $\l \cap \n_\g = 0$, we have $\l \cong (\n_\g \subsetplus \l )/ \n_\g = ( \n_\g \subsetplus \l_0 )/ \n_\g \cong \l_0$.  Thus $\l$ is also locally semisimple, and any two Levi components of $\g$ are isomorphic.  

Now suppose that $\tilde{\l}$ is any locally semisimple subalgebra with $\l \subset \tilde{\l} \subset \g$.  Since $\tilde{\l}$ is locally semisimple, $\tilde{\l} \subset [\g,\g]$ and $\tilde{\l} \cap \r = 0$.  Therefore $(\r \cap [\g,\g]) \subsetplus \l \subset (\r \cap [\g,\g]) \subsetplus \tilde{\l} \subset [\g,\g]$.  Since $(\r \cap [\g,\g]) \subsetplus \l = [\g,\g]$, it follows that $(\r \cap [\g,\g]) \subsetplus \l = (\r \cap [\g,\g]) \subsetplus \tilde{\l}$, and hence $\l = \tilde{\l}$.  Therefore $\l$ is a maximal locally semisimple subalgebra of $\g$.
\end{proof}

Note: we do not know whether an arbitrary maximal locally semisimple subalgebra of a subalgebra of $\gl_\infty$ is a Levi component.

\begin{cor} \label{ssconditions}
Let $\g$ be a finitary Lie algebra.  The following conditions on $\g$ are equivalent:
\begin{enumerate}
\item \label{condition1} $\g$ is locally semisimple;
\item \label{condition3} $\g$ is isomorphic to a direct sum of finite-dimensional simple Lie algebras and copies of $\sl_\infty$, $\so_\infty$, and $\sp_\infty$; 
\item \label{condition4} $\g$ is equal to its own Levi component;
\item \label{condition5}  $\g = [\g,\g]$ and $\r= 0$.
\end{enumerate}
\end{cor}

\begin{proof}
The equivalence of conditions (\ref{condition1}) and (\ref{condition3}) was already quoted from \cite{DP4} in Theorem~\ref{locsemisimple}.  
The equivalence of conditions (\ref{condition1}) and (\ref{condition4}) follows from Theorem~\ref{maxlocss}.

It is straightforward to check that  (\ref{condition1}) implies (\ref{condition5}).  Conversely, suppose that $\g = [\g,\g]$ and $\r= 0$.  Since the linear nilradical $\n_\g$ is by definition contained in $\r$, we have $\n_\g = 0$.    There exists a Levi component $\l \subset \g$, and one has $\g = [\g,\g] = \n_\g \subsetplus \l = \l$.  Thus $\g$ is  locally semisimple, and (\ref{condition5}) implies (\ref{condition1}).  
\end{proof}

We next turn to splittable subalgebras of $\gl(V,V_*)$.  We need preliminary material related to locally reductive finitary Lie algebras.

\begin{thm} \label{finitary}
Suppose a finitary Lie algebra $\g$ is a union of reductive subalgebras.  Then there exists an injective homorphism $\varphi \colon \g \hookrightarrow \z(\g) \oplus \bigoplus_i \g_i$ with $\g_i$ isomorphic to $\gl_\infty$, $\so_\infty$, $\sp_\infty$, or a finite-dimensional simple Lie algebra, such that $\varphi|_{\z(\g)} = \mathrm{id}$ and $\varphi([\g,\g])= \bigoplus_i[\g_i , \g_i]$.
\end{thm}

\begin{proof}
Note first that there is a subalgebra $\k \subset \g$ such that $\k$ is a union of reductive subalgebras, $\z(\k) = 0$, and $\g = \z(\g) \oplus \k$.  Indeed, let $\g = \bigcup_i \g_i$ be an exhaustion of $\g$ by (finite-dimensional) reductive Lie algebras.  For each $i$, one has $\z(\g) \cap \g_i \subset \z(\g_i)$.  Hence one may inductively choose nested reductive subalgebras $\k_i \subset \g_i$ such that $\g_i = (\z(\g) \cap \g_i) \oplus \k_i$.    Then $\g = \z(\g) \oplus \k$, where $\k = \bigcup_i \k_i$.

For the rest of the proof we assume $\z(\g) = 0$.  This implies $\n_\g = 0$ via Lemma~\ref{locreductivenilradical}.  Let $\F$, $\G$ be a taut couple as given by Theorem~\ref{tautcouple}. Then $\g \subset \p$ and $\n_\g = \n_\p \cap \g$, where $\p := \St_\F \cap \St_\G$.  Moreover, since $\n_\g = 0$, we have an injective homomorphism $\g \hookrightarrow {\p} / {\n_\p} = \bigoplus_{\gamma \in C} \gl\big({F''_\gamma} / {F'_\gamma} , {G''_\gamma} / {G'_\gamma}\big)$, where the notation is as in Proposition~\ref{splitexact}.  Since $\g$ has trivial center, the homomorphism $\varphi \colon \g \hookrightarrow ({\p} / {\n_\p}) / \z({\p} / {\n_\p})$ is also injective.  

One has $$({\p} / {\n_\p}) / \z({\p} / {\n_\p}) =  \bigoplus_{\gamma \in C_0} \sl\big({F''_\gamma} / {F'_\gamma} , {G''_\gamma} / {G'_\gamma}\big) \oplus  \bigoplus_{\gamma \in C \setminus C_0} \gl\big({F''_\gamma} / {F'_\gamma} , {G''_\gamma} / {G'_\gamma}\big),$$ where $C_0$ is the set of $\gamma \in C$ such that ${F''_\gamma} / {F'_\gamma}$ is finite dimensional.  
By Proposition~\ref{glirreducibleeverywhere} $[\g,\g]$ acts irreducibly on ${F''_\gamma} / {F'_\gamma}$ and ${G''_\gamma} / {G'_\gamma}$ for all $\gamma \in C$.  Proposition~\ref{irreducibleeverywhere} implies that $\varphi([\g,\g])$ is locally semisimple.  

Let $\k_0$ denote the direct sum of finite-dimensional direct summands of $\varphi([\g,\g])$, and let $\k_i$ for $i \in I$ be the infinite-dimensional direct summands of $\varphi([\g,\g])$. 
Proposition~\ref{irreducibleeverywhere} implies
that $C \setminus C_0$ is a disjoint union of finite subsets $C_i$ such that 
$\k_i \subset \bigoplus_{\gamma \in C_i} \sl\big({F''_\gamma} / {F'_\gamma} , {G''_\gamma} / {G'_\gamma}\big)$.  Furthermore, if $i \in I$ and $\gamma \in C_i$, then the projection of $\k_i$ to $\sl\big({F''_\gamma} / {F'_\gamma} , {G''_\gamma} / {G'_\gamma}\big)$ yields an isomorphism of $\k_i$ with $\sl\big({F''_\gamma} / {F'_\gamma}, {G''_\gamma} / {G'_\gamma}\big)$, $\so\big({F''_\gamma} / {F'_\gamma}\big)$, or $\sp\big({F''_\gamma} / {F'_\gamma}\big)$, where in the final two cases one has an identification of ${F''_\gamma} / {F'_\gamma}$ and ${G''_\gamma} / {G'_\gamma}$.

Consider the projections $\pi_i \colon ({\p} / {\n_\p} ) / \z({\p} / {\n_\p}) \rightarrow \bigoplus_{\gamma \in C_i}  \gl \big({F''_\gamma} / {F'_\gamma} , {G''_\gamma} / {G'_\gamma} \big)$ and $\pi_0 \colon ({\p} / {\n_\p} ) / \z({\p} / {\n_\p}) \rightarrow \bigoplus_{\gamma \in C_0}  \sl \big({F''_\gamma} / {F'_\gamma} , {G''_\gamma} / {G'_\gamma} \big)$.    We have
$$\bigoplus_{i \in I \sqcup \{ 0 \} } \k_i \subset \varphi (\g) \subset \bigoplus_{i \in I \sqcup \{0\} } \pi_i  \circ \varphi ( \g).$$ 
Moreover, $[\pi_i \circ \varphi (\g) , \pi_i \circ \varphi (\g)] = \pi_i \circ \varphi ([\g ,\g]) = \k_i$.  This shows in particular that $ \pi_i \circ \varphi ( \g)$ is contained in the normalizer of $\k_i$ for $i \in I \sqcup \{0\}$.  Lemma~\ref{infinitenormalizer} implies that the normalizer of $\pi_i \circ \varphi (\g)$ for $i \in I$ is equal to   $\k_i $ unless $\k_i \cong \sl_\infty$, in which case the normalizer of $\pi_i \circ \varphi (\g)$ is the diagonal copy of $\gl_\infty$ containing $\k_i$.   Observe that $\pi_0 \circ \varphi (\g)$ is a union of reductive subalgebras and acts irreducibly on the quotients ${F''_\gamma} / {F'_\gamma}$ and ${G''_\gamma} / {G'_\gamma}$ for all $\gamma \in C_0$.  By finite-dimensional Lie theory, $\pi_0 \circ \varphi (\g) \subset \bigoplus_{\gamma \in C_0}  \sl \big({F''_\gamma} / {F'_\gamma} , {G''_\gamma} / {G'_\gamma} \big)$ is locally semisimple.  Therefore $\pi_0 \circ \varphi (\g) = [\pi_0 \circ \varphi (\g) , \pi_0 \circ \varphi (\g) ] = \pi_0 \circ \varphi ( [\g,\g]) = \k_0$.

We therefore take $\g_0$ to be $\k_0$, and we take $\g_i$ to be the normalizer in $\bigoplus_{\gamma \in C_i}  \gl \big({F''_\gamma} / {F'_\gamma} , {G''_\gamma} / {G'_\gamma} \big)$ of $\k_i$ for $i \in I$.  One may check that $[\g_i, \g_i] = \k_i$ for $i \in I \sqcup \{0\}$, and we have 
$$\bigoplus_i [\g_i, \g_i] \subset \varphi(\g) \subset \bigoplus_i \g_i.$$
It follows from the above inclusions that $\varphi ([\g,\g]) = \bigoplus_i [\g_i , \g_i]$.
\end{proof}

\begin{cor} \label{unionreductive}
A finitary Lie algebra which is a union of reductive subalgebras is locally reductive.
\end{cor}

The following theorem strengthens Theorem \ref{existence} under the assumption that $\g$ is splittable.

\begin{thm} \label{locreductivepart}
Suppose $\g \subset \gl(V,V_*)$ is a splittable subalgebra.  Then there exists a locally reductive subalgebra $\g_{red}$, called a \emph{locally reductive part}, of $\g$ such that $\g = \n_\g \subsetplus \g_{red}$.  Furthermore, $\l := [\g_{red},\g_{red}]$ is a Levi component of $\g$, and there exists a toral subalgebra $\t \subset \g$ such that $\g_{red} = \l \subsetplus \t$.
\end{thm}

\begin{proof}
Let $\pi : \g \rightarrow {\p} / {\n_\p} = \bigoplus_{\gamma \in C} \gl  \big(F''_\gamma / F'_\gamma , G''_\gamma / G'_\gamma \big)$ be as in the proof of Theorem~\ref{existence}.  
Since $\g$ acts irreducibly on $F''_\gamma  / F'_\gamma$ and $G''_\gamma / G'_\gamma$ for all $\gamma \in C$, Proposition~\ref{irreducibleeverywhere} implies the existence of
an exhaustion $ \bigoplus_{\gamma \in C} \gl  \big(F''_\gamma / F'_\gamma , G''_\gamma / G'_\gamma \big) = \bigcup_{j \in \Z_{\geq 0}} \q_j$, where each $\q_j$ is a finite-dimensional Lie algebra isomorphic to a direct sum of general linear Lie algebras such that $\pi(\g) \cap \q_j$ acts irreducibly on the natural and conatural representations of each direct summand of $\q_j$.

Define finite-dimensional splittable subalgebras $\g_j \subset \g$ with $\pi(\g_j) = \q_j$ inductively as follows.  Suppose one is given a finite-dimensional splittable subalgebra $\g_{j-1}$ with $\pi(\g_{j-1}) =  \q_{j-1}$.  Since $\pi^{-1} (\q_j)$ is a splittable subalgebra of $\g$, there exists a finite-dimensional splittable subalgebra $\g_j \subset \pi^{-1} (\q_j)$ containing $\g_{j-1}$ with $\pi(\g_j) = \q_j$.  

We next show that $\n_{\g_j} = \n_\g \cap \g_j$ for each $j$.  The containment $ \n_\g \cap \g_j \subset \n_{\g_j}$ follows from Lemma~\ref{nilradicalcontainment}.  The image $\pi(\g_j)$ in $\q_j$ acts irreducibly on the natural and conatural representations of $\q_j$.  Hence any nilpotent ideal of $\pi(\g_j)$ is contained in the center of $\q_j$.  But every nilpotent element of $\pi^{-1} (\z(\q_j))$ is in $\n_\g$, and the claim $\n_{\g_j} = \n_\g \cap \g_j$ follows.  As a result, $\n_{\g_{j-1}} \subset \n_{\g_j}$.

We now choose inductively subalgebras $\m_j \subset \g_j$ such that $\m_{j-1} \subset \m_j$ and $\m_j$ is maximal among the subalgebras of $\g_j$ which act semisimply on $V$.
We claim that $\m_j$ is a reductive part of $\g_j$.  This follows from Theorem 4.1 of \cite{Mostow}, which asserts that
any two subalgebras of a Lie algebra $\k \subset \gl_n$ maximal among those that act semisimply on the natural representation of $\gl_n$ are conjugate under an inner automorphism from the radical of $[\k,\k]$. As $\g_j$ is splittable, it has some reductive part $({\g_j})_{red}$, i.e.\ $\g_j = \n_{\g_j} \subsetplus ({\g_j})_{red}$.  Because $(\g_j)_{red}$ is maximal among the subalgebras of $\g_j$ which act semisimply on $V$, it is conjugate to $\m_j$, and hence the latter is also a reductive part of $\g_j$.

Let $\g_{red} := \bigcup_j \m_j$.  By Corollary~\ref{unionreductive}, $\g_{red}$ is locally reductive.  
The fact that $\g_j = \n_{\g_j} \subsetplus \m_j$ for all $j$ implies $\g = \n_\g + \bigcup_j \g_j = \n_\g \subsetplus \g_{red}$.  Note that  $\l := [\g_{red} , \g_{red}]$ is a Levi component of $\g$.
Let $\t_j$ be nested maximal toral subalgebras of $\m_j$, and take $\t := \bigcup_j \t_j$.  As $\m_j = [\m_j , \m_j] + \t_j$ for each $j$, we have $\g_{red} = \l + \t$.  Let $\t'$ be a vector space complement of $\t \cap \l$ in $\t$.  Then $\g_{red} = \l \subsetplus \t'$.
\end{proof}

\begin{defn} \label{traceconditions}
Let $\g$ be a splittable subalgebra of  $\gl(V,V_*)$.  A subalgebra $\k \subset \g$ is said to be \emph{defined by trace conditions} on $\g$ if $$\n_\g + [\g,\g] \subset \k.$$
\end{defn}

That is, a subalgebra $\k$ of a splittable subalgebra $\g$ of $\gl_\infty$ is defined by trace conditions if and only if $\k$ contains a Levi component of $\g$ and the linear nilradical $\n_\g$.

\begin{prop} \label{toraltrace}
Let $\g$ be a splittable subalgebra of  $\gl(V,V_*)$, and $\k$ a subalgebra defined by trace conditions on $\g$.  Then $\k$ is splittable.  Furthermore, $\g$ and $\k$ have the same linear nilradical and Levi components.
\end{prop}

\begin{proof}
Let $\g_{red}$ be a reductive part of $\g$, as given in Theorem~\ref{locreductivepart}.   By Theorem~\ref{locreductivepart}, there exists a toral subalgebra $\t \subset \g$ such that $\g_{red} = \l \subsetplus \t$, where $\l := [\g_{red},\g_{red}]$ is a Levi component of $\g$.  By the definition of $\k$, we have $\n_\g \subsetplus \l \subset \k \subset \g$.  Hence $\k$ admits a vector space decomposition $\k = \n_\g \oplus \l \oplus \t'$, where $\t' := \t \cap \k$.  As $\k$ is generated by splittable subalgebras, \cite[Ch 7 \S 5 Cor 1]{Bourbaki} implies that $\k$ is splittable.
The last statement is straightforward to check, so we omit this.
\end{proof}

For example, take $\g :=  \Big\{ \left( \begin{array}{cc} A & 0 \\ 0 & B \end{array} \right) ~\colon  A, B \in \gl(V,V_*)  \Big\} \subset \gl(V \oplus V, V_* \oplus V_*)$. The linear nilradical of $\g$ is trivial, and $[\g,\g] = \sl(V,V_*) \oplus \sl(V,V_*)$ is the unique Levi component of $\g$.  The subalgebra $$\k := \Big\{ \left( \begin{array}{cc} A & 0 \\ 0 & B \end{array} \right) \in \g ~\colon \tr \, A = 2 \tr \,B \Big\} \subset \g$$ is defined by trace conditions on $\g$.

\section{Parabolic subalgebras of $\gl_\infty$ and $\sl_\infty$} \label{glparabolic}

We are now ready to start the discussion of parabolic subalgebras of $\gl_\infty$ and $\sl_\infty$.  As in the finite-dimensional case, we define a subalgebra $\p$ of a finitary Lie algebra $\k$ to be \emph{parabolic} if there exists a Borel (that is, a maximal locally solvable) subalgebra $\b$ of $\k$ with $\b \subset \p$.  Recall that, for any parabolic subalgebra $\p_n$ of $\g_n = \gl_n$ or $\g_n = \sl_n$, the following statements hold:
\begin{itemize}
\item $\p_n$ is the stabilizer of a unique flag $\F_n$ in the natural representation of $\g_n$;
\item $\p_n$ is self-normalizing in $\g_n$;
\item $\p_n = \n_n \subsetplus \m_n$, where $\n_n$ is the linear nilradical of $\p_n$ and $\m_n$ is a subalgebra of $\g_n$ which is reductive in $\g_n$;
\item $\p_n = (\n_n)^\perp$, where the perpendicular complement is taken with respect to a nondegenerate invariant form on $\g_n$. 
\end{itemize}

In the case of $\g = \gl(V,V_*)$ or $\g = \sl(V,V_*)$, the above statements admit generalizations and yield in general a chain of three potentially different parabolic subalgebras of $\g$, namely $\p \subset N_\g(\p) \subset (\n_\p)^\perp$.  

Borel subalgebras of $\gl(V,V_*)$ were understood in \cite{DP2} using the concept of a closed generalized flag.  More precisely, any Borel subalgebra is the stabilizer of a unique maximal closed generalized flag in $V$.  The following theorem strengthens this result by providing some alternative descriptions of maximal closed generalized flags and their stabilizers.

\begin{thm} \label{tfae}
Let $\F$ be a semiclosed generalized flag in $V$. The following are equivalent:
\begin{enumerate}
\item \label{TFAEone} $\F$ is a maximal semiclosed generalized flag;
\item \label{TFAEtwo} $\F$ is a maximal closed generalized flag;
\item \label{TFAEthree} $\St_\F$ is a Borel subalgebra of $\gl(V,V_*)$ or $\sl(V,V_*)$;
\item \label{TFAEfour} $\St_\F$ is a minimal parabolic subalgebra of $\gl(V,V_*)$ or $\sl(V,V_*)$;
\item \label{TFAEfive} there exists a (unique) maximal semiclosed generalized flag $\G$ in $V_*$ such that $\St_\F = \St_\G$.
\end{enumerate}
Furthermore, if $\F$, $\G$ are as in (\ref{TFAEfive}), then they form a taut couple.
\end{thm}

\begin{proof}
Let $\F = \{ F'_\alpha, F''_\alpha \}_{\alpha \in A}$ be a maximal semiclosed generalized flag.  We now show that $\F$ is a closed generalized flag.  By Lemma~\ref{maximalsemiclosed}, $\dim F''_\alpha / F'_\alpha = 1$ for all $\alpha \in A$ such that $\overline{F'_\alpha} = F'_\alpha$.  Fix $\alpha \in A$.  If $\overline{F'_\alpha} = F'_\alpha$, then $F''_\alpha$ is closed as it contains a closed subspace of finite codimension.  If $\overline{F'_\alpha} = F''_\alpha$, then $F''_\alpha$ is also closed.  This proves that (\ref{TFAEone}) implies (\ref{TFAEtwo}).  To see that (\ref{TFAEtwo}) implies (\ref{TFAEone}), we notice that if $\F$ is a maximal closed generalized flag, then again $\dim F''_\alpha / F'_\alpha = 1$ for all $\alpha \in A$ with $\overline{F'_\alpha} = F'_\alpha$ \cite{DP2}.

It is shown in \cite{DP2} and \cite{D} that (\ref{TFAEtwo}) and (\ref{TFAEthree}) are equivalent.  The equivalence of (\ref{TFAEthree}) and (\ref{TFAEfour}) follows directly from the definition of a parabolic subalgebra.

We note next that (\ref{TFAEthree}) implies (\ref{TFAEfive}).  Indeed, if $\St_\F$ is a Borel subalgebra of $\gl(V,V_*)$, then it follows from \cite{DP2} that $\St_\F = \St_\G$ for a unique maximal closed generalized flag $\G$ in $V_*$.  We showed above that $\G$ is maximal also as a semiclosed generalized flag. 

The implication of (\ref{TFAEthree}) from (\ref{TFAEfive}) requires no further argument, due to the symmetry of $V$ and $V_*$.  Finally, since the chain $\F^\perp$ is stable under $\St_\F$, the equality $\St_\F = \St_\G$ implies that $\F$, $\G$ form a taut couple.  Thus the proof is complete.
\end{proof}

Let $\F = \{F'_\alpha , F''_\alpha \}_{\alpha \in A}$ and $\G= \{G'_\beta, G''_\beta \}_{\beta \in B}$ be semiclosed generalized flags in $V$ and $V_*$, respectively, and assume $\F$, $\G$ form a taut couple.  For $\alpha \in A$ and $\beta \in B$, define $\alpha < \beta$ if $\langle F''_\alpha , G''_\beta \rangle = 0$.
This gives a strict ordering on $A \cup_C B$ extending the orderings of $A$ and $B^{op}$ (where the superscript ${op}$ indicates opposite ordering).  For any $\alpha \in A$, we have 
$$(F''_\alpha)^\perp = \bigcup_{\beta \in B, \, G''_\beta \subset (F''_\alpha)^\perp} G''_\beta = \bigcup_{\alpha < \beta \in B} G''_\beta.$$
Proposition~\ref{splitexact} (\ref{firstitem}) therefore implies the formula
$$\n_{\St_\F \cap \St_\G} = \sum_{A \ni \alpha < \beta \in B} F''_\alpha \otimes G''_\beta,$$
and consequently Proposition~\ref{splitexact} (\ref{seconditem}) yields 
$$\St_\F \cap \St_\G = \sum_{A \ni \alpha \leq \beta \in B} F''_\alpha \otimes G''_\beta.$$

We need the following two technical lemmas.

\begin{lemma} \label{orderingproperty}
For any $\alpha \in A$, one has
$\overline{F'_\alpha} = (\bigcup_{\alpha \leq \beta} G''_\beta)^\perp$.
\end{lemma}

\begin{proof}
If $\alpha \notin C$, then $\overline{F'_\alpha} = \overline{F''_\alpha} = (F''_\alpha)^{\perp \perp} = ( \bigcup_{\alpha < \beta} G''_\beta)^\perp = ( \bigcup_{\alpha \leq \beta} G''_\beta)^\perp$.  If $\alpha \in C$, then $\overline{F'_\alpha} = F'_\alpha = (G''_\alpha)^\perp = (\bigcup_{\alpha \leq \beta} G''_\beta)^\perp$.  
\end{proof}

\begin{lemma}\label{recoverslinfty}
Let $\b \subset \sl_\infty$ be a Borel subalgebra. Then no proper subalgebra of $\sl_\infty$ contains $[\n_\b, \sl_\infty]$.
\end{lemma}

\begin{proof}
Let $\F = \{F'_\alpha, F''_\alpha \}_{\alpha \in A}$ and $\G = \{ G'_\beta , G''_\beta \}_{\beta \in B}$ be the maximal closed generalized flags in $V$ and $V_*$, respectively, such that $\b = \St_\F = \St_\G$ \cite{DP2}.
Define 
$$W := \bigcup_{\alpha \in A} F'_\alpha$$
and
$$W_* := \bigcup_{\beta \in B} G'_\beta.$$

We will show first that $\sl(V,V_*) \cap (W \otimes V_* + V \otimes W_*) \subset [\n_\b , \sl(V,V_*)]$.  Let $x \in W$ and $y \in V_*$ satisfy $\langle x,y \rangle = 0$.  Since $x \in W$, there exists $\alpha \in A$ such that $x \in F''_\alpha \setminus F'_\alpha$ and $(F''_\alpha)^\perp \neq 0$.  Let $0 \neq z \in (F''_\alpha)^\perp$, and let $w \in V$ be any element such that $\langle w, z \rangle \neq 0$.  Then $[ x \otimes z, w \otimes y] =  \langle w , z \rangle x \otimes y$.  Since $x \otimes z \in \n_\b$ and $w \otimes y  \in \gl(V,V_*)$, this implies $x \otimes y \in [\n_\b , \gl(V,V_*)]$.  Moreover, there must exist a copy of $\gl_n \subset \gl(V,V_*)$ such that $x$, $y$, $z$, and $w$ are all elements of $\gl_n$, and hence one may replace $w \otimes y$ with a traceless element whose commutator with $x \otimes z$ is unchanged.  This shows that  $x \otimes y \in [\n_\b , \sl(V,V_*)]$.
Similarly, if $x \in V$, $y \in W_*$, and $\langle x,y \rangle = 0$, then $x \otimes y \in [\n_\b , \sl(V,V_*)]$.  

Since $\sl(V,V_*) \cap (W \otimes V_* + V \otimes W_*)$ is spanned by elements of the form $x \otimes y \in V \otimes V_*$ such that $\langle x, y \rangle = 0$ and $x \in W$ or $y \in W_*$, it follows that $\sl(V,V_*) \cap (W \otimes V_* + V \otimes W_*) \subset [\n_\b , \sl(V,V_*)]$. 

If $W \neq V$, then $W \subset V$ is a pair in $\F$.  Hence $\dim V/ W \leq 1$ if $W$ is closed, and $\overline{W} = V$ if $W$ is not closed.  Similarly,  $\dim V_* / W_* \leq 1$ if $W_*$ is closed, and $\overline{W_*} = V_*$ if $W_*$ is not closed.  

This implies that for any $x \in V$ and $y \in V_*$ with $\langle x , y \rangle = 0$, there exist $u \in y^\perp \cap W$ and $v \in x^\perp \cap W_*$ with $\langle u , v \rangle \neq 0$.  Then $x \otimes v \in \sl(V,V_*) \cap (V \otimes W_*)$ and $u \otimes y \in \sl(V,V_*) \cap (W \otimes V_*)$ and $[x \otimes v, u \otimes y] = \langle u , v \rangle x \otimes y$.  This shows that if $\k$ is any subalgebra with $[\n_\b , \sl(V,V_*)] \subset \k \subset \sl(V,V_*)$, then
 $x \otimes y \in \k$. Since $\sl(V,V_*)$ is spanned by elements of this form, we have shown that $\k = \sl(V,V_*)$.
\end{proof}

In order to state the main theorem of this section, we need two more preliminary constructions. 

Suppose $\F$, $\G$ form a taut couple.  
We define a subalgebra $(\St_\F \cap \St_\G)_-$ of $\sl(V,V_*)$ or $\gl(V,V_*)$ as follows.  Recall from Proposition~\ref{splitexact} that there is a homomorphism (with kernel $\n_{\St_\F \cap \St_\G}$):  $$\St_\F \cap \St_\G \longrightarrow \bigoplus_{\gamma \in C} \gl \big({F''_\gamma} / {F'_\gamma} , {G''_\gamma} / {G'_\gamma} \big).$$
Let $C_0$ be the set of $\gamma \in C$ for which $\dim {F''_\gamma} / {F'_\gamma} < \infty$. Define $(\St_\F \cap \St_\G)_-$ as the preimage of 
\[ 
\bigoplus_{\gamma \in C \setminus C_0} \sl \big({F''_\gamma} / {F'_\gamma} , {G''_\gamma} / {G'_\gamma}
 \big) \oplus \bigoplus_{\eta \in C_0}  \gl \big( {F''_\eta} / {F'_\eta} , {G''_\eta} / {G'_\eta}
 \big)
\]
in $\St_\F\cap\St_\G$. Then for any set of vector spaces $V_\alpha$ and $(V_\beta)_*$ as in Proposition~\ref{splitexact}, one has the subalgebra of $\gl(V,V_*)$
$$(\St_\F \cap \St_\G)_- = \n_{\St_\F \cap \St_\G} \subsetplus  \left( \bigoplus_{\gamma \in C \setminus C_0} \sl \big( V_\gamma, (V_\gamma)_* \big) \oplus \bigoplus_{\eta \in C_0}  \gl \big( V_\eta , (V_\eta)_* \big) \right).$$ 

The second construction produces a new taut couple $\F^c$, $\G^c$ from a given taut couple $\F$, $\G$.

\begin{lemma} \label{nonclosedsuccessor}
Let $\F$ be a semiclosed generalized flag.  Let $F' \subset F''$ be a pair such that $F''$ is not closed. Then $F'' \subset \overline{F''}$ is a pair in $\F$.  
\end{lemma}

\begin{proof}
Suppose for the sake of a contradiction that $F''$ does not have an immediate successor in $\F$.  Then $F''$ must be the intersection of those subspaces of $\F$ properly containing $F''$.  By the definition of a semiclosed generalized flag, each pair contains a closed subspace.  Hence $F''$ is the intersection of all closed subspaces in $\F$ containing $F''$, and is itself closed.
\end{proof}

\begin{prop} \label{primeflags}
For any semiclosed generalized flag $\F$, there exists a closed generalized flag $\F^c$ which is maximal among the closed generalized flags arising as subchains of $\F$.  Furthermore, if $\F$, $\G$ form a taut couple, then $\F^c$, $\G^c$ also form a taut couple. 
\end{prop}

\begin{proof}
Lemma~\ref{nonclosedsuccessor} implies that $\F^{\perp \perp}$ is a subchain of $\F$.  Let $\F^c$ be the generalized flag obtained from $\F^{\perp \perp} \cup \{0,V_*\}$ via the general procedure described in Section~\ref{tautcouples}.  Then $\F^c$ is a closed generalized flag which refines any closed generalized flag contained in $\F$.  Explicitly, if $F''_\alpha$ is not closed, then the two pairs $F'_\alpha \subset F''_\alpha$ and $F''_\alpha \subset \overline{F''_\alpha}$ in $\F$ are replaced by the single pair $F'_\alpha \subset \overline{F''_\alpha}$ in $\F^c$.  Finally, if $\G^\perp$ is stable under $\St_\F$, then $(\G^c)^\perp = \G^\perp$ is stable under $\St_{\F^c} = \sum_{\alpha \in A} \overline{F''_\alpha} \otimes (F'_\alpha)^\perp$. 
\end{proof}

The following theorem is our main result in this section.

\begin{thm} \label{slglcharacterization}
Let $\g$ be one of $\gl(V,V_*)$ and $\sl(V,V_*)$.

\begin{enumerate}
\item Let $\p \subset \g$ be a vector subspace.  Then $\p$ is a parabolic subalgebra if and only if there exists a (unique) taut couple $\F$, $\G$ such that $$(\St_\F \cap \St_\G)_- \subset \p \subset \St_\F \cap \St_\G.$$

\item Let $\p \subset \g$ be a parabolic subalgebra, and let $\p_+ := \St_\F \cap \St_\G$ and $\p_- := (\St_\F \cap \St_\G)_- $.  Then the following statements hold.
\begin{enumerate}
\item \label{parta}
$\p$ is splittable and hence (by Theorem \ref{locreductivepart}) admits a decomposition $\p = \n_\p \subsetplus \p_{red}$.
\item \label{partb}
$\p_+ = N_\g(\p_+) = N_\g(\p_-) = N_\g(\p)$, and $\n_{\p_+} = \n_{\p_-} = \n_\p$.
\item \label{partc}
One has $\p_+ \subset \p' := (\n_\p)^\perp$, where the orthogonal complement is taken with respect to the form $\tr(XY)$ on $\g$.  In fact $\p' = \St_{\F^c} \cap \St_{\G^c}$, where $\F^c$ and $\G^c$ are the closed generalized flags associated to $\F$ and $\G$ as above in Proposition~\ref{primeflags}.  Furthermore, $\n_{\p'} = \n_\p$. 
\end{enumerate}
\end{enumerate}
\end{thm}

\begin{proof}
Suppose $\p \subset \g$ is a parabolic subalgebra.  Let $\F$, $\G$ be a $\p$-stable taut couple as given by Theorem~\ref{tautcouple}.  Set $\p_+ := \St_\F \cap \St_\G$.  Let $\varphi \colon \p \rightarrow {\p_+} / { \n_{\p_+}}$ be the inclusion of $\p$ into $\p_+$ followed by the quotient map.  Recall from Proposition~\ref{splitexact} that $${\p_+} / { \n_{\p_+}} \cong \bigoplus_{\gamma \in C}\gl\big({F''_\gamma} / {F'_\gamma} , {G''_\gamma} / {G'_\gamma}\big).$$  
Let $\pi_\gamma \colon {\p_+} / { \n_{\p_+}} \rightarrow\gl\big({F''_\gamma} / {F'_\gamma} , {G''_\gamma} / {G'_\gamma}\big)$ denote the quotient map for each $\gamma \in C$.  
Clearly $$\bigoplus_{\gamma \in C} \varphi(\b) \cap \sl\big({F''_\gamma} / {F'_\gamma} , {G''_\gamma} / {G'_\gamma}\big) \subset
\varphi (\p) \subset \bigoplus_{\gamma \in C} \pi_\gamma \circ \varphi (\p).$$

Fix $\gamma \in C$.  We now show that $\varphi (\b) \cap \sl\big({F''_\gamma} / {F'_\gamma} , {G''_\gamma} / {G'_\gamma}\big)$ is a Borel subalgebra of $\sl\big({F''_\gamma} / {F'_\gamma} , {G''_\gamma} / {G'_\gamma}\big)$ for any Borel subalgebra $\b$ of $\g$ contained in $\p$.  We know from the classification of Borel subalgebras in \cite{DP2} quoted in Theorem~\ref{tfae} that $\b$ is the stabilizer of a maximal closed generalized flag $\H$ in $V$.  Furthermore $\H$ is a refinement of $\F$ by Lemma~\ref{stablesubspaces}.  The intersection $\varphi (\b) \cap \sl\big({F''_\gamma} / {F'_\gamma} , {G''_\gamma} / {G'_\gamma}\big)$ is the stabilizer in $\sl\big({F''_\gamma} / {F'_\gamma} , {G''_\gamma} / {G'_\gamma}\big)$ of $\{ (H \cap F''_\gamma) / F'_\gamma \colon H \in \H \}$.  One may check that  $\{ (H \cap F''_\gamma) / F'_\gamma : H \in \H \}$ is a maximal closed generalized flag in ${F''_\gamma} / {F'_\gamma}$ under the pairing ${F''_\gamma} / {F'_\gamma} \times {G''_\gamma} / {G'_\gamma} \rightarrow \C$.  Hence by the same classification $\varphi (\b) \cap \sl\big({F''_\gamma} / {F'_\gamma} , {G''_\gamma} / {G'_\gamma}\big)$ is a Borel subalgebra of $\sl\big({F''_\gamma} / {F'_\gamma} , {G''_\gamma} / {G'_\gamma}\big)$.

Since $\b \subset \p$, this means that $\pi_\gamma \circ \varphi (\p)$ is a subalgebra of $\gl\big({F''_\gamma} / {F'_\gamma} , {G''_\gamma} / {G'_\gamma}\big)$ which acts irreducibly on both ${F''_\gamma} / {F'_\gamma}$ and ${G''_\gamma} / {G'_\gamma}$ and contains a Borel subalgebra of $\sl\big({F''_\gamma} / {F'_\gamma} , {G''_\gamma} / {G'_\gamma}\big)$.  By Theorem \ref{irreducible}, $\pi_\gamma\circ \varphi (\p)$ is either $\sl\big({F''_\gamma} / {F'_\gamma} , {G''_\gamma} / {G'_\gamma}\big)$ or $ \gl\big({F''_\gamma} / {F'_\gamma} , {G''_\gamma} / {G'_\gamma}\big)$.

For each $\gamma \in C$, one has 
\begin{eqnarray*}
[\varphi (\b) \cap \sl\big({F''_\gamma} / {F'_\gamma} , {G''_\gamma} / {G'_\gamma}\big) ,   \sl\big({F''_\gamma} / {F'_\gamma} , {G''_\gamma} / {G'_\gamma}\big) ]
& \subset &[\varphi (\b) \cap 
\sl\big({F''_\gamma} / {F'_\gamma} , {G''_\gamma} / {G'_\gamma}\big)
, \pi_\gamma \circ \varphi (\p) ] \\
&= &[\varphi (\b) \cap 
\sl\big({F''_\gamma} / {F'_\gamma} , {G''_\gamma} / {G'_\gamma}\big)
, \varphi( \p) ] \\
& \subset & \varphi(\p).
\end{eqnarray*}
It follows from Lemma~\ref{recoverslinfty} that
$\sl\big({F''_\gamma} / {F'_\gamma} , {G''_\gamma} / {G'_\gamma}\big)
\subset \varphi(\p)$.  

The proof of Theorem~\ref{existence} implies the existence of subalgebra $\l \subset \p$ which is a Levi component of both $\p$ and $\p_+$, since in either case $\l$ is the pullback of $\bigoplus_{\gamma \in C} \sl\big({F''_\gamma} / {F'_\gamma} , {G''_\gamma} / {G'_\gamma}\big)$.  Note that $\n_{\p_+} \subset \b \subset \p$.  Hence $[\p_+ , \p_+] = \n_{\p_+} \subsetplus \l \subset \p$.

Let $V_\alpha$ and $(V_\beta)_*$ for $\alpha \in A$ and $\beta \in B$ be as in Proposition~\ref{splitexact} (\ref{seconditem}).  Then $\p_+ = \n_{\p_+} \subsetplus \bigoplus_{\gamma \in C} \gl \big( V_\gamma , (V_\gamma)_* \big)$, and we have shown that $ \n_{\p_+} \subsetplus \bigoplus_{\gamma \in C} \sl \big( V_\gamma , (V_\gamma)_* \big) \subset \p$.  Since $\p$ contains $\b$, one has
$$\bigoplus_{\eta \in C_0} \sl \big( V_\eta , (V_\eta)_* \big) + \b \cap \bigoplus_{\eta \in C_0} \gl \big( V_\eta , (V_\eta)_* \big) = \g \cap \bigoplus_{\eta \in C_0} \gl \big( V_\eta , (V_\eta)_* \big) \subset \p.$$
Hence we have shown $(\St_\F \cap \St_\G)_- \subset \p$.

As the commutator subalgebra of $\St_\F \cap \St_\G$ is contained in $(\St_\F \cap \St_\G)_-$, any intermediate vector subspace is a subalgebra.  It remains to show that  $(\St_\F \cap \St_\G)_-$ is a parabolic subalgebra of $\gl(V,V_*)$.  Let $\tilde{\F}$, $\tilde{\G}$ be a maximal taut couple refining the taut couple $\F$, $\G$, with the following property.  For each $\gamma \in C \setminus C_0$, and every immediate predecessor-successor pair $H' \subset H''$ of $\tilde{\F}$ with $F'_\gamma \subset H' \subset H'' \subset F''_\gamma$, one requires $\overline{H'} = H''$.  This is possible due to an example given in \cite{DP2} of a locally nilpotent Borel subalgebra of $\gl_\infty$.  As noted in Theorem~\ref{tfae}, the subalgebra $\St_{\tilde{\F}} \cap \St_{\tilde{\G}}$ is a Borel subalgebra of $\g$.  Of course we have $\St_{\tilde{\F}} \cap \St_{\tilde{\G}} \subset \St_\F \cap \St_\G$ since the one couple refines the other, and one may check from the construction that indeed $\St_{\tilde{\F}} \cap \St_{\tilde{\G}}  \subset (\St_\F \cap \St_\G)_-$.

\begin{itemize}
\item[(\ref{parta})] The subalgebra $\p_+ = \St_\F \cap \St_\G$ is splittable because the stabilizer of any chain is a splittable subalgebra, and the intersection of splittable subalgebras is splittable.  Since $\p$ is defined by trace conditions on $\p_+$, Proposition~\ref{toraltrace}  implies that $\p$ is splittable.

\item[(\ref{partb})]
Proposition~\ref{toraltrace} implies $\n_\p =  \n_{\p_+}$.  Using the fact that $\p_-$ is a parabolic subalgebra with $(\p_-)_+ = \p_+$, we see that $\n_{\p_-} = \n_{(\p_-)_+} = \n_{\p_+}$.

We will show that $N_\g (\p) = \p_+$.  The result for the normalizer of $\p_+$ and $\p_-$ follows from the fact that $(\p_-)_+ = (\p_+)_+ = \p_+$.

We have already seen that $[\p_+ , \p_+] \subset \p_-$.  Thus $\p_+$ normalizes $\p$.  We show below that $[X , \p_-] \subset \p_+$ implies $X \in \p_+$. 
As a result, $[X , \p] \subset \p$ implies $[X , \p_-] \subset [X, \p] \subset \p \subset \p_+$, which in turn implies $X \in \p_+$.

Suppose that $[X , \p_-] \subset \p_+$.  Then $X = X' + \sum_{i=1}^n v_i \otimes w_i$ for some $X' \in \p_+$ and $0 \neq v_i \in V_{\alpha_i}$ and $0 \neq w_i \in (V_{\beta_i})_*$ with $\alpha_i > \beta_i$ for $i \in \{1 , \ldots n \}$.  Thus $[\sum_{i=1}^n v_i \otimes w_i, \p_-] \subset \p_+$, and indeed
 $[\sum_{i=1}^n v_i \otimes w_i, \p_+] \subset \p_+$.  

Assume for the sake of a contradiction that $n \geq 1$.  Assume without loss of generality that $\alpha_1 \geq \alpha_i$ for all $i$, and that $\beta_1 \leq \beta_i$ if $\alpha_i = \alpha_1$.
We may also assume that the vectors $v_i$ for which $\alpha_i = \alpha_1$ are linearly independent.  Observe $v_1 \otimes \bigcup_{b \geq \alpha_1} G''_b \subset \p_+$.  For any $y \in \bigcup_{b \geq \alpha_1} G''_b$,
compute
$$\left[ \sum_i v_i \otimes w_i , v_1 \otimes y \right] = \sum_i (\langle v_1 , w_i \rangle v_i \otimes y - \langle v_i , y \rangle v_1 \otimes w_i).$$
By the linear independence of $v_1$ from the other vectors $v_i$ with $\alpha_i = \alpha_1$, we see that $v_1 \otimes w_1$ must appear with coefficient zero.  Therefore $\langle v_1 , y \rangle = 0$.  This shows that $v_1 \in (\bigcup_{\alpha_1 \leq b} G''_b)^\perp = \overline{F'_\alpha}$, by Lemma~\ref{orderingproperty}.  Consequently $\alpha_1 \in A \setminus C$.  One may show similarly that $\beta_1 \in B \setminus C$.

We next show that there exists $a \in A$ such that $\beta_1 < a < \alpha_1$.  Since $\langle F''_{\alpha_1} , G''_{\beta_1} \rangle \neq 0$, one has that $G''_{\beta_1} \nsubseteq (F''_{\alpha_1})^\perp = (F'_{\alpha_1})^\perp$.  Therefore $\langle F'_{\alpha_1} , G''_{\beta_1} \rangle \neq 0$.  Since $F'_{\alpha_1} = \bigcup_{a < {\alpha_1}} F''_\alpha$, there exists $a < {\alpha_1}$ such that $\langle F''_a , G''_{\beta_1} \rangle \neq 0$.  Therefore $\beta_1 \leq a$, and since $\beta_1 \notin C$, indeed $\beta_1 < a$.

Now assume that the vectors $w_i$ for which $\beta_i = \beta$ are linearly independent, relaxing the above linear independence hypothesis.  
A similar line of argument shows $v_1 \in (\bigcup_{a \leq b} G''_b)^\perp = \overline{F'_a}$.  Hence $v_1 \in F''_a$, which contradicts the fact that $v_1 \in F''_\alpha \setminus F'_\alpha$ and $a < \alpha$.  As a result, $n = 0$ and $X = X' \in \p_+$.

\item[(\ref{partc})] 
To check that $\p_+$ pairs trivially with $\n_\p$, recall that $\p_+ = \n_\p \subsetplus \bigoplus_{\gamma \in C} V_\gamma \otimes (V_\gamma)_*$ and $\n_\p = \sum_{\alpha \in A} F''_\alpha \otimes (F''_\alpha)^\perp$ by Proposition~\ref{splitexact}.  For any $\alpha$, $a \in A$, the pairing of $F''_\alpha \otimes (F''_\alpha)^\perp$ with $F''_a \otimes (F''_a)^\perp$ is $\langle F''_\alpha , (F''_a)^\perp \rangle \langle F''_a , (F''_\alpha)^\perp \rangle = 0$.  For any $\alpha \in A$ and $\gamma \in C$,  the pairing of $F''_\alpha \otimes (F''_\alpha)^\perp$ with $F''_\gamma \otimes G''_\gamma$ is $\langle F''_\alpha , G''_\gamma \rangle \langle F''_\gamma , (F''_\alpha)^\perp \rangle = 0$.

We omit the proof that $\p' = \St_{\F^c} \cap \St_{\G^c}$.

Proposition~\ref{splitexact} (\ref{firstitem}) gives $\n_{\p_+} = \n_{\St_\F \cap \St_\G} = \sum_{\alpha \in A} F''_\alpha \otimes (F''_\alpha)^\perp$.  
Suppose $F''_\alpha$ is not closed.  Then $F'_\alpha \subset F''_\alpha$ and $F''_\alpha \subset \overline{F''_\alpha}$ are two pairs in $\F$, and these yield terms $F''_\alpha \otimes (F''_\alpha)^\perp$ and $\overline{F''_\alpha} \otimes ( \overline{F''_\alpha} )^\perp$ in the expression for $\n_{\p_+} $.  Note that $F''_\alpha \otimes (F''_\alpha)^\perp \subset \overline{F''_\alpha} \otimes ( \overline{F''_\alpha} )^\perp$.
By Proposition~\ref{primeflags}, the generalized flag $\F^c$ has the pair $F'_\alpha \subset \overline{F''_\alpha}$, which gives rise to the term $\overline{F''_\alpha} \otimes (\overline{F''_\alpha})^\perp$ in the analogous expression for $\n_{\p'} = \n_{\St_{\F^c} \cap \St_{\G^c}}$.  Since the remaining pairs of $\F$ and $\F^c$ are identical, the expressions for  $\n_{\p'}$ and $\n_{\p_+}$ agree, and hence $\n_{\p'} = \n_{\p_+} = \n_\p$.
\end{itemize}

To get the uniqueness of the couple $\F$, $\G$, consider that the normalizer of $\p$ in $\gl(V,V_*)$ is $\p_+ = \St_\F \cap \St_\G$.  Hence the normalizer of $\p$ is in the image of the map $(\F , \G) \mapsto \St_\F \cap \St_\G$ taking taut couples to subalgebras of $\gl(V,V_*)$. Proposition~\ref{uniquetautpair} states that this map is injective, and the uniqueness of the couple follows.
\end{proof}

\begin{cor}\label{selfnormalizingparabolics}
The map $$(\F,\G) \mapsto \St_\F \cap \St_\G$$ is a bijection from the set of taut couples in $V$ and $V_*$ to the set of self-normalizing parabolic subalgebras of $\gl(V,V_*)$.
\end{cor}

The following proposition describes the closed subspaces of $V$ which are stable under the subalgebra $(\St_\F \cap \St_\G)_- \subset \gl(V,V_*)$ for a taut couple $\F$, $\G$.  The proof is similar to that of Lemma~\ref{stablesubspaces}, and we leave it to the reader.

\begin{prop}
Let $\k := (\St_\F \cap \St_\G)_- \subset \gl(V,V_*)$.  The set of closed $\k$-stable subspaces of $V$ is a chain (containing the chain of closed subspaces of $\F$).
\end{prop}

\begin{thm}
Let $\p \subset \gl(V,V_*)$ be any subalgebra.  If $\p = \n_\p^\perp$, then $\p$ is a parabolic subalgebra.  (The converse does not hold, since $\n_\p^\perp = \p'$ for any parabolic subalgebra $\p$.)
\end{thm}

\begin{proof}
Suppose  $\p \subset \gl(V,V_*)$ is a subalgebra with the property $\p = \n_\p^\perp$.  Let $\F$, $\G$ be the couple given by Theorem~\ref{tautcouple}, and set $\q := \St_\F \cap \St_\G$.  Theorem~\ref{tautcouple} gives that $\p \subset \q$ and $\n_\p = \n_\q \cap \p$.  By Theorem~\ref{slglcharacterization}, $\q$ is a parabolic subalgebra of $\gl(V,V_*)$, and $\q \subset \n_\q^\perp$.  This shows that $\q \subset \n_\q^\perp \subset \n_\p^\perp \subset \p$.  Therefore $\p = \q$ is a parabolic subalgebra.
\end{proof}

It is worth pointing out that the stabilizer of a single semiclosed generalized flag $\F$ in $V$ (or $\G$ in $V_*$) is a parabolic subalgebra of $\gl(V,V_*)$.   Via the general construction described in Section~\ref{preliminaries}, the chain $\F^\perp \cup \{0,V_*\}$ in $V_*$ yields a semiclosed generalized flag $\G$.  One may check that $\F$, $\G$ form a taut couple, and the joint stabilizer equals $\St_\F$.  Alternatively, Theorem~\ref{tautcouple} applied to the subalgebra $\k = \St_\F$ produces a taut couple, one generalized flag of which can be taken to be $\F$.  Then the statement $\St_\F \subset \St_\F \cap \St_\G$ implies $\St_\F = \St_\F \cap \St_\G$.

\section{Parabolic subalgebras of $\so_\infty$ and $\sp_\infty$}\label{secParsosp}

In this section we take $\g$ to be $\so(V)$ or $\sp(V)$, under a suitable identification of $V$ and $V_*$.  The statements in this section are given without proof, as they are similar to those of the corresponding statements for $\gl_\infty$.  

Suppose $\F = \{F'_\alpha , F''_\alpha \}_{\alpha \in A}$ is a self-taut generalized flag in $V$.  
For each $\alpha \in A$ such that $F'_\alpha$ is closed and coisotropic, there exists a unique pair $F'_\beta \subset F''_\beta$ in $\F$ such that $F'_\beta = (F''_\alpha)^\perp$.  Then $F'_\beta$ is closed and coisotropic, and setting $f(\alpha) := \beta$, we obtain a map $f \colon \{ \alpha \in A ~\colon F'_\alpha \textrm{ is closed, }F''_\alpha \textrm{ is isotropic} \} \rightarrow \{ \alpha \in A ~\colon F'_\alpha \textrm{ is closed, coisotropic} \}$.

\begin{prop} \label{isodefinec}
The map $f$ is a bijection
$$\{ \alpha \in A ~\colon F'_\alpha \textrm{ is closed, }F''_\alpha \textrm{ is isotropic} \} \rightarrow \{ \alpha \in A ~\colon F'_\alpha \textrm{ is closed, coisotropic} \}.$$
\end{prop}

We fix the notation $C := \{ \alpha \in A ~\colon F'_\alpha \textrm{ is closed, }F''_\alpha \textrm{ is isotropic} \}$.
For any $\gamma \in C$, let $G'_\gamma := F'_{f(\gamma)}$ and $G''_\gamma := F''_{f(\gamma)}$.
For each $\gamma \in C$, one has $G'_\gamma = (F''_\gamma)^\perp$ and $F'_\gamma = (G''_\gamma)^\perp$.

As a corollary of Lemma~\ref{stablesubspaces}, every subspace of a self-taut generalized flag is either isotropic or coisotropic.  To see this, consider that the definition of self-taut implies that $X^\perp$ is stable under the $\gl(V,V)$-stabilizer of $\F$ for any $X \in \F$.  By Lemma~\ref{stablesubspaces}, $X^\perp \cup \F$ is a chain, so either $X \subset X^\perp$ or $X^\perp \subset X$.  Define $F$ as the union of all isotropic subspaces $F''_\alpha$ for $\alpha \in A$, and $G$ as the intersection of all coisotropic subspaces $F'_\alpha$ for $\alpha \in A$.  Clearly $F \subset G$.  We claim furthermore that $F  \neq G$ implies $F \subset G$ is a pair in $\F$.  Indeed, if $v$ is any vector in $G \setminus F$, then the unique pair in $\F$ given by property (ii) of the definition of a generalized flag must be $F \subset G$.

Consider the maps
\begin{eqnarray*}
\Lambda \colon \gl(V,V) & \rightarrow & {\bigwedge}^2 V = \so(V) \\
v \otimes w & \mapsto & v \otimes w - w \otimes v
\end{eqnarray*}
and
\begin{eqnarray*}
\mathrm{S} \colon \gl(V,V) & \rightarrow & \Sym^2 (V) = \sp(V) \\
v \otimes w & \mapsto & v \otimes w + w \otimes v.
\end{eqnarray*}

\begin{prop}
If $\p := \St_{\F,\g} \subset \g$, then the following statements hold.
\begin{enumerate}
\item $\n_\p = \sum_{\alpha \in A} F''_\alpha \wedge (F''_\alpha)^\perp$ if $\g = \so(V)$,
and $\n_\p = \sum_{\alpha \in A} F''_\alpha \& (F''_\alpha)^\perp$ if $\g = \sp(V)$.
\item  There exist vector subspaces $V_\alpha \subset V$ for $\alpha \in A$  such that
\begin{itemize}
\item $F''_\alpha = F'_\alpha \oplus V_\alpha$ for each $\alpha \in A$;
\item $V = \bigoplus_{\alpha \in A} V_\alpha$;
\item $\langle  \bigoplus_{\gamma \in C} (V_\gamma)_* , \bigoplus_{\gamma \in C} (V_\gamma)_* \rangle = 
\langle W , \bigoplus_{\gamma \in C} V_\gamma \rangle = \langle W , \bigoplus_{\gamma \in C} (V_\gamma)_* \rangle = 0$;
\item $\langle V_\gamma , (V_\eta)_* \rangle = 0$ for distinct $\gamma \neq \eta \in C$;
\end{itemize}
where $(V_\gamma)_*$ denotes the chosen vector space complement of $G'_\gamma$ in $G''_\gamma$ for any $\gamma \in C$, and where $W$ denotes the chosen vector space complement to $F$ in $G$ if $F \neq G$ and otherwise $W = 0$.

Moreover,
$$\p = \n_\p \subsetplus \left( \so(W) \oplus \bigoplus_{\gamma \in C} \Lambda \big( \gl \big( V_\gamma, (V_\gamma)_* \big) \big) \right)$$ if $\g = \so(V)$,
and
$$\p = \n_\p \subsetplus \left(\sp(W) \oplus \bigoplus_{\gamma \in C} \mathrm{S} \big( \gl \big( V_\gamma, (V_\gamma)_* \big) \big) \right)$$ if $\g = \sp(V)$.
\end{enumerate}
\end{prop}

Note that although $\Lambda$ and $\mathrm{S}$ are not Lie algebras homomorphisms, their restrictions to $\gl \big( V_\gamma, (V_\gamma)_* \big)$ are homomorphisms.

The following theorem is the proper analogue of Theorem~\ref{tfae}.  A generalized flag $\G$ in $V$ is called \emph{isotropic} (resp.\ \emph{coisotropic}) if every proper nontrivial subspace of $V$ appearing in $\G$ is isotropic (resp.\ coisotropic).

\begin{thm}
Let $\G$ be an isotropic semiclosed generalized flag in $V$.  If $\g = \sp(V)$, the following are equivalent:
\begin{enumerate}
\item \label{eins} $\G$ is a maximal semiclosed isotropic generalized flag;
\item \label{zwei} $\G$ is a maximal closed isotropic generalized flag;
\item \label{drei} $\St_{\G,\g}$ is a Borel subalgebra of $\g$; 
\item \label{vier} $\St_{\G,\g}$ is a minimal parabolic subalgebra of $\g$;
\item \label{funf} there exists a  
maximal semiclosed coisotropic generalized flag $\H$ in $V$ such that $\St_{\H,\g} = \St_{\G,\g}$.
\end{enumerate}
If $\g = \so(V)$, then (\ref{eins}) $\Leftrightarrow$ (\ref{zwei}) $\Rightarrow$ (\ref{drei}) $\Leftrightarrow$ (\ref{vier}) $\Rightarrow$ (\ref{funf}).
\end{thm}

If $\g = \so(V)$, then it is possible to realize a Borel subalgebra of $\g$ as the stabilizer of a non-maximal closed isotropic generalized flag.  Thus the listed conditions are not equivalent in the $\so(V)$ case. 
For a more general discussion of this phenomenon involving parabolic subalgebras, see \cite{DPW}.

\begin{thm}
A self-taut generalized flag $\F$ is maximal self-taut if and only if $\F$ is maximal semiclosed.
\end{thm}

\begin{lemma}
Let $\b \subset \g$ be a Borel subalgebra.  Then no proper subalgebra of $\g$ contains $[\n_\b,\g]$.
\end{lemma}

We define a subalgebra $(\St_{\F,\g})_- \subset \St_{\F,\g}$ as follows.   Since $\F$ is self-taut, it forms a taut couple with itself, and we have already defined the subalgebra $(\St_\F)_- \subset \gl(V,V)$.  We take $(\St_{\F,\g})_- := (\St_\F)_- \cap \g$.
One has
$$(\St_{\F,\g})_- = \n_{\St_{\F,\g}} \subsetplus \left( 
 \so(W) \oplus \bigoplus_{\gamma \in C_0} \Lambda \big( \gl \big( V_\gamma, (V_\gamma)_* \big) \big) \oplus \bigoplus_{\gamma \in C \setminus C_0} \Lambda \big( \sl \big( V_\gamma, (V_\gamma)_* \big) \big)  \right)$$
if $\g = \so(V)$ and
$$(\St_{\F,\g})_- = \n_{\St_{\F,\g}} \subsetplus \left( \sp(W) \oplus \bigoplus_{\gamma \in C_0} \mathrm{S} \big( \gl \big( V_\gamma, (V_\gamma)_* \big) \big) \oplus \bigoplus_{\gamma \in C \setminus C_0} \mathrm{S} \big( \sl \big( V_\gamma, (V_\gamma)_* \big) \big) \right)$$
if $\g = \sp(V)$, where $C_0$ denotes the set of $\gamma \in C$ for which $\dim V_\gamma < \infty$.

\begin{thm} \label{sospcharacterization}
Let $\g$ be one of $\so(V)$ and $\sp(V)$.

\begin{enumerate}
\item Let $\p \subset \g$ be a vector subspace.  Then $\p$ is a parabolic subalgebra if and only if there exists a (unique) self-taut generalized flag $\F$ in $V$ such that $$(\St_{\F,\g})_- \subset \p \subset \St_{\F,\g}.$$  

\item Let $\p \subset \g$ be a parabolic subalgebra, and let $\p_+ := \St_{\F,\g}$ and $\p_- := (\St_{\F,\g})_- $.  Then the following statements hold.
\begin{enumerate}
\item $\p$ is splittable and hence (by Theorem \ref{locreductivepart}) admits a decomposition $\p = \n_\p \subsetplus \p_{red}$.
\item $\p_+ = N_\g(\p_+) = N_\g(\p_-) = N_\g(\p)$, and $\n_{\p_+} = \n_{\p_-} = \n_\p$.
\item One has $\p_+ \subset \p' := (\n_\p)^\perp$, where the orthogonal complement is taken with respect to the form $\tr(XY)$ on $\g$.  In fact $\p' = \St_{\F^c , \g}$, where $\F^c$ is the closed generalized flag associated to $\F$ as in Proposition~\ref{primeflags}.  Furthermore, $\n_{\p'} = \n_\p$. 
\end{enumerate}
\end{enumerate}
\end{thm}

Note that $\F^c$ is self-taut if $\F$ is a self-taut generalized flag.

\section{Parabolic subalgebras of more general finitary Lie algebras} \label{generalparabolic}

In this section we describe all parabolic subalgebras of a splittable finitary Lie subalgebra of $\gl_\infty$. We start with the case when the linear nilradical is trivial.  In this case we do not need the splittability assumption and thus simply consider finitary Lie algebras which are unions of reductive subalgebras.  By Corollary~\ref{unionreductive}, this is equivalent to considering locally reductive finitary Lie algebras.

\begin{thm} \label{finitaryborel}
Let $\g$ be a locally reductive finitary Lie algebra.  If $\b \subset \g$ is a Borel subalgebra, then $\b \cap [\g,\g]$ is a Borel subalgebra of $[\g,\g]$.
\end{thm}

\begin{proof}
By Theorem~\ref{finitary}, there is a decomposition $[\g,\g] = \bigoplus_i \s_i$ and an injective homomorphism $\varphi \colon \g \hookrightarrow \z(\g) \oplus \bigoplus_i \g_i$, where each Lie algebra $\g_i$ is isomorphic to $\gl_\infty$, $\so_\infty$, $\sp_\infty$, or a finite-dimensional simple Lie algebra, such that $\varphi|_{\z(\g)} = \mathrm{id}$ and $[\g_i,\g_i] = \varphi(\s_i)$.  
Let $\tilde{\b}$ be a Borel subalgebra of $\z(\g) \oplus \bigoplus_i \g_i$ containing $\varphi(\b)$.  As $\b$ is contained in the locally solvable subalgebra $\varphi^{-1} (\tilde{\b})$, the maximality of $\b$ implies $\b = \varphi^{-1}(\tilde{\b})$.  

Since Borel subalgebras respect direct sums, $\tilde{\b} \cap \g_i$ is a Borel subalgebra of $\g_i$.  If $\g_i$ differs from $\varphi(\s_i)$, then $\g_i \cong \gl_\infty$.  In this case, by the classification of Borel subalgebras of $\gl_\infty$ and $\sl_\infty$ seen in \cite{D}, we know that $\tilde{\b} \cap [\g_i , \g_i]$ is a Borel subalgebra of $[\g_i , \g_i]$.  This enables us to conclude that $\tilde{\b} \cap [\g_i,\g_i]$ is a Borel subalgebra for all $i$, and thus $\varphi^{-1} (\tilde{\b}) \cap [\g,\g]$ is a Borel subalgebra of $[\g,\g]$.  Finally, the observation that $\varphi^{-1} (\tilde{\b}) \cap [\g,\g] = \b \cap [\g,\g]$ finishes the proof.
\end{proof}

The following theorem shows that the parabolic subalgebras of a finitary Lie algebra exhausted by reductive subalgebras can be understood in terms of the parabolic subalgebras of $\gl_\infty$, $\so_\infty$, $\sp_\infty$, and finite-dimensional simple Lie algebras.  For any parabolic subalgebra $\p$ of $\sl_\infty$, $\gl_\infty$, $\so_\infty$, or $\sp_\infty$, the normalizer of $\p$ is the parabolic subalgebra $\p_+$.  Since parabolic subalgebras of finite-dimensional simple Lie algebras are self-normalizing, we use for convenience the notational convention $\p_+ := \p$ for any parabolic subalgebra $\p$ of a finite-dimensional simple Lie algebra.

\begin{thm} \label{generalthm}
Let $\g$ be a locally reductive finitary Lie algebra, and $\varphi \colon \g \hookrightarrow \z(\g) \oplus \bigoplus_i \g_i$ be an injective homomorphism as in Theorem~\ref{finitary}.

\begin{enumerate}
\item
Let $\p \subset \g$ be a vector subspace.  Then $\p$ is a parabolic subalgebra if and only if there exist
parabolic subalgebras $\p_i \subset \g_i$ such that
$$\varphi^{-1} \left( \z(\g) \oplus \bigoplus_i (\p_i)_- \right) \subset \p \subset \varphi^{-1} \left( \z(\g) \oplus \bigoplus_i (\p_i)_+ \right).$$

\item
Let $\p \subset \g$ be a parabolic subalgebra, and let $\p_+ := \varphi^{-1} \left( \z(\g) \oplus \bigoplus_i (\p_i)_+ \right)$  and $\p_- := \varphi^{-1} \left( \z(\g) \oplus \bigoplus_i (\p_i)_- \right) $.
Then the following statements hold.
\begin{enumerate}
\item \label{firstclaim} If $\g$ is a splittable subalgebra of $\gl_\infty$, then $\p$ is also splittable.
\item \label{secondclaim} $\p_+ = N_\g(\p_+) = N_\g(\p_-) = N_\g(\p)$, and $\n_{\p_+} = \n_{\p_-} = \n_\p$.
\item \label{thirdclaim} One has $\p_+ \subset \p' := (\n_\p)^\perp$, where the orthogonal complement is taken with respect to the form $\tr(XY)$ on $\g$.  In fact $\p' = \varphi^{-1} \left( \z(\g) \oplus \bigoplus_i (\p_i)' \right)$. 
Furthermore, $\n_{\p'} = \n_\p$. 
\end{enumerate}
\end{enumerate}
\end{thm}

\begin{proof} 
We continue to use the notation of the proof of Theorem~\ref{finitaryborel}.  
Let $\p \subset \g$ be a parabolic subalgebra.  There exists a Borel subalgebra $\b \subset \p$, and by Theorem~\ref{finitaryborel}, $\b \cap \s_i$ is a Borel subalgebra of $\s_i$.  Hence $\p \cap \s_i$ is a parabolic subalgebra of $\s_i$.

Let $\pi_i \colon \z(\g) \oplus \bigoplus_j \g_j \rightarrow \g_i$ denote the projection.  Then 
$$\z(\g) \oplus \bigoplus_i \p \cap \s_i \subset \p \subset \varphi^{-1} \left( \z(\g) \oplus  \bigoplus_i \pi_i( \varphi(\p)) \right).$$
Observe that $[\varphi(X) , \pi_i (\varphi(\p))] = \varphi([X , \p])$ for any $X \in \s_i$.  Hence $[\varphi(\p \cap \s_i) , \pi_i(\varphi(\p)) ] = [\varphi(\p \cap \s_i , \varphi (\p) ] \subset \varphi (\p)$.  Furthermore, $[\varphi(\p \cap \s_i) , \pi_i (\varphi(\p))] \subset \varphi(\s_i)$ because $[\varphi(\s_i) , \tilde{\g} ] \subset \varphi(\s_i)$.  The injectivity of $\varphi$ implies $\varphi(\p \cap \s_i) = \varphi(\p) \cap \varphi(\s_i)$.  Thus $[\varphi(\p \cap \s_i) , \pi_i (\varphi(\p))]  \subset \varphi(\p \cap \s_i)$.  Hence $\pi_i (\varphi(\p))$ is contained in the normalizer in $\g_i$ of $\varphi(\p \cap \s_i)$.  By Proposition~\ref{sospcharacterization}, the normalizer of $\varphi(\p \cap \s_i)$ is $(\varphi(\p \cap \s_i))_+$ if $\s_i$ is isomorphic to $\so_\infty$ or $\sp_\infty$.  If $\g_i$ is simple, we define $\p_i = \p \cap \s_i$.  If $\g_i$ is not simple, then $\s_i \cong \sl_\infty$, and we take $\p_i$ to be the (self-normalizing) parabolic subalgebra of $\g_i$ which is the stabilizer of the taut couple related to the parabolic subalgebra $\p \cap \s_i$, and in this case $\p_i$ is the normalizer in $\g_i$ of $\p \cap \s_i$.  For all $i$ we have $\pi_i (\varphi(\p)) \subset (\p_i)_+$.  Also note that $\varphi^{-1}(\p_i)_- \subset \p \cap \s_i$.  Thus we have shown
$$\varphi^{-1} \left( \z(\g) \oplus \bigoplus_i (\p_i)_- \right) \subset \p \subset \varphi^{-1} \left( \z(\g) \oplus \bigoplus_i (\p_i)_+ \right).$$

Conversely, fix parabolic subalgebras $\p_i$ of $\g_i$.  Observe that the commutator subalgebra of $\varphi^{-1} \big( \z(\g) \oplus \bigoplus_i (\p_i)_+ \big)$ is contained in $\varphi^{-1} \big( \z(\g) \oplus \bigoplus_i (\p_i)_- \big)$, hence any intermediate vector subspace $\p$ is indeed a subalgebra.  
To show that $\p$ is a parabolic subalgebra, it suffices to show that $\varphi^{-1} \big( \z(\g) \oplus \bigoplus_i (\p_i)_- \big)$ is a parabolic subalgebra of $\g$.  It is evident that $\z(\g) \oplus \bigoplus_i (\p_i)_-$ contains a Borel subalgebra $\b$ of $\z(\g) \oplus \bigoplus_i \g_i$.  Since $\b$ is the unique extension of $\b \cap (\z(\g) \oplus \bigoplus_i [\g_i,\g_i])$ to a Borel subalgebra, we see that $\varphi^{-1} (\b)$ is a Borel subalgebra of $\g$.  As a result, $\varphi^{-1} \big( \z(\g) \oplus \bigoplus_i (\p_i)_- \big)$ is a parabolic subalgebra of $\g$.

\begin{itemize}
\item[(\ref{firstclaim})]  Suppose $\g$ is a splittable subalgebra of $\gl_\infty$.  The stabilizer in $\g$ of any chain of subspaces in a representation of $\g$ is a splittable subalgebra of $\gl_\infty$.  Since the  intersection of splittable subalgebras is splittable, it follows that $\varphi^{-1} \big( \bigoplus_i (\p_i)_+ \big)$ is splittable.  By Proposition~\ref{toraltrace}, a subalgebra defined by trace conditions on  $\varphi^{-1} \big( \bigoplus_i (\p_i)_+ \big)$ is also splittable.

Note furthermore that $\z(\g)$ is splittable.  Indeed, for any $X \in \g$, its Jordan components $X_{ss}$ and $X_{nil}$  are polynomials in $X$.  So $X \in \z(\g)$ implies $X_{ss}$, $X_{nil} \in \z(\g)$.  As $\p$ is generated by $\z(\g)$ and a subalgebra defined by trace conditions on $\varphi^{-1} \big( \bigoplus_i (\p_i)_+ \big)$, it is itself splittable by \cite[Ch 7 \S 5 Cor 1]{Bourbaki}.
\end{itemize}
We omit the proof of parts (\ref{secondclaim}) and (\ref{thirdclaim}).
\end{proof}

Suppose $\g$ is a locally semisimple finitary Lie algebra.  Then $\g \cong \bigoplus_i \s_i$, where each $\s_i$ is $\sl_\infty$, $\so_\infty$, $\sp_\infty$, or a finite-dimensional simple Lie algebra.  It does not follow from Theorem \ref{generalthm} that for any parabolic subalgebra $\p$ of $\g$ there exist parabolic subalgebras $\p_i \subset \s_i$ such that $\p = \bigoplus_i \p_i$.  Indeed, set $\g := \sl(V \oplus V , V_* \oplus V_*) \oplus \sl(V \oplus V , V_* \oplus V_*)$ considered as a subalgebra of $\gl(V \oplus V \oplus V \oplus V, V_* \oplus V_* \oplus V_* \oplus V_*)$.  Consider the parabolic subalgebra $\p$ defined as the elements of $\g$ with block decomposition
$$\left( \begin{array}{cccc} 
A & B & 0 & 0 \\
0 & C & 0 & 0 \\
0 & 0 & D & E \\
0 & 0 & 0 & F
\end{array} \right)$$
such $\tr A = \tr D = -\tr C = - \tr F$.
Then $\p_-$ consists of all elements of $\g$ with the same block decomposition satisfying the conditions $\tr A = \tr C = \tr D = \tr F = 0$.  Let $X$ be a semisimple element of $\gl(V,V_*)$ with nonzero trace.  One may observe that
$$\p = \p_- \subsetplus \Span \left\{ 
\left( \begin{array}{rrrr} 
X & 0 & 0 & 0 \\
0 & -X & 0 & 0 \\
0 & 0 & X & 0 \\
0 & 0 & 0 & -X
\end{array} \right) \right\}$$
and that $\p$ is not the direct sum of parabolic subalgebras of the direct summands of $\g$.

Using the language of parabolic subalgebras, we can state the following generalization of the Karpelevi\v c Theorem, which asserts that a maximal subalgebra of a simple finite-dimensional Lie algebra must be semisimple or parabolic \cite{Karpelevic}.

\begin{cor}
Any maximal subalgebra of $\gl_\infty$, $\sl_\infty$, $\so_\infty$, or $\sp_\infty$ is either a direct sum of simple subalgebras or a parabolic subalgebra.
\end{cor}

\begin{proof}
The maximal subalgebras of the above Lie algebras are described explicitly in \cite{DP4}.  
A maximal subalgebra of $\gl_\infty$ is either the commutator subalgebra, which is simple, or the stabilizer of a single subspace in $V$ or $V_*$.  If a maximal subalgebra of $\sl_\infty$ is not isomorphic to one of $\so_\infty$ and $\sp_\infty$, which are simple, then it is again the stabilizer of a single subspace in $V$ or $V_*$.  A maximal subalgebra of $\so_\infty$ or $\sp_\infty$ is either the direct sum of two simple subalgebras or the stabilizer of a single subspace in $V$.  At the end of Section~\ref{glparabolic} it was noted that the stabilizer in $\gl_\infty$ or $\sl_\infty$ of a single semiclosed generalized flag in $V$ or $V_*$ is parabolic.  In particular, the stabilizer of a single subspace of $V$ or $V_*$ is a parabolic subalgebra.  Analogously, the stabilizer in $\so_\infty$ or $\sp_\infty$ of a single isotropic or coisotropic subspace of $V$ is a parabolic subalgebra.
\end{proof}

We conclude this section by describing the parabolic subalgebras of any splittable finitary Lie algebra.

\begin{thm} \label{mostgeneralthm}
Let $\g$ be a splittable subalgebra of $\gl(V,V_*)$.  Fix a locally reductive part $\g_{red} \subset \g$ according to Theorem~\ref{locreductivepart}.
Then the map $$\p \mapsto \n_\g \subsetplus \p$$
is a bijection between the set of parabolic subalgebras of $\g_{red}$ and the set of parabolic subalgebras of $\g$, under which Borel subalgebras correspond to Borel subalgebras.  
\end{thm}

\begin{proof}
Any Borel subalgebra of $\g$ contains the locally solvable radical of $\g$, so any Borel subalgebra of $\g$ contains $\n_\g$.  It follows that any parabolic subalgebra of $\g$ contains $\n_\g$.  Hence the map  $\p \mapsto \n_\g \subsetplus \p$
is a bijection between the set of parabolic subalgebras of $\g_{red}$ and the set of parabolic subalgebras of $\g$.  
\end{proof}

\appendix

\section{Appendix: Cartan subalgebras of splittable Lie algebras}

In the existing literature only Cartan subalgebras of locally finite Lie algebras admitting an exhaustion by reductive Lie algebras has been studied, see \cite{DPS} and the references therein. In this appendix we extend the theory of Cartan subalgebras to arbitrary splittable subalgebras of locally reductive Lie algebras. 

If $\h$ is a subalgebra of a locally reductive Lie algebra, let $\h_{ss}$ denote the set of semisimple Jordan components of the elements of $\h$.
For any subset $\a \subset \g$ and any subalgebra $\k \subset \g$, we define the centralizer of $\a$ in $\k$, denoted $\z_\k (\a)$, to be the set of elements of $\k$ which commute in $\g$ with all elements of $\a$.
For an arbitrary Lie algebra $\h \subset \k$, we define $\overline{\k^0(\h)}$ as the subalgebra of $\k$ consisting of all elements of $\k$ on which every finite-dimensional subalgebra $\h_{fin}$ of $\h$ acts locally nilpotently.

The following are generalizations of Proposition 3.1, Theorem 3.2 and Lemma 3.3 from \cite{DPS}.

\begin{prop} \label{centralizer}
Let $\h$ be a locally nilpotent subalgebra of a splittable subalgebra $\k$ of a locally reductive Lie
algebra.  Then the following assertions hold:
\begin{enumerate}
\item $\h \subseteq \z_\k(\h_{ss})$;
\item \label{toral} $\h_{ss}$ is a toral subalgebra of $\g$;
\item  \label{self-normalizing} $\z_\k (\h_{ss})$ is a self-normalizing subalgebra of $\k$.
\end{enumerate}
\end{prop}

\begin{proof}
Let $h, h' \in \h$.  The local nilpotence of $\h$ implies $(\ad
h)^n(h') = 0$ for some $n$.  Since $\ad h_{ss}$ is a polynomial in $\ad h$ with no constant term, it follows that $(\ad h_{ss})(\ad h)^{n-1}(h') = 0$.  Because an element commutes with its semisimple part, $(\ad h)^{n-1}(\ad h_{ss})(h') = 0$, and it follows by induction that $(\ad h_{ss})^n(h') = 0$.  Hence $(\ad h_{ss})(h') = 0$.  Since $\k$ is splittable, $\h_{ss} \subset \k$, and we have shown $\h \subseteq \z_\k(\h_{ss})$.

Furthermore, by the same argument, $(\ad h')(h_{ss}) = 0$ implies $(\ad h'_{ss})(h_{ss}) = 0$.  Therefore any two elements of $\h_{ss}$ commute.  Since the sum of any two commuting semisimple elements is semisimple, $\h_{ss}$ is a subalgebra.

Finally, suppose $x$ is in the normalizer of $ \z_\k(\h_{ss})$.  For any $y \in \h_{ss}$, we have that $[x,y] \in \z_\k(\h_{ss})$.  Thus $[[x,y],y]=0$, and as $y$ is semisimple it follows that $[x,y]=0$. Hence $x \in \z_\k(\h_{ss})$, i.e.\ $\z_\k(\h_{ss})$ is self-normalizing.
\end{proof}

\begin{thm} \label{main}
Let $\k$ be a splittable subalgebra of a locally reductive Lie algebra, and $\h$ a subalgebra of $\k$.  The following conditions on $\h$ are equivalent:
\begin{enumerate}
\item \label{D} $\h = \z_\k(\h_{ss})$;
\item  \label{E} $\h = \z_\k(\t)$ for some maximal toral subalgebra $\t \subseteq \k$;
\item \label{F} $\h = \overline{\k^0(\h)}$.
\end{enumerate}
In addition, any subalgebra satisfying one of the above conditions is locally nilpotent, splittable, and self-normalizing.
\end{thm}

\begin{lemma}\label{lemma1}
If $\h$ is locally nilpotent and splittable, then $\overline{\k^0(\h)} = \z_\k(\h_{ss})$.
\end{lemma}
\begin{proof}
By \cite[Ch. VII, \S 5, Prop. 5]{Bourbaki} $\h = \h_{ss} \oplus \h_{nil}$, with $\h_{nil}$ being the subalgebra of all nilpotent elements in $\h$.  It follows that  $\overline{\k^0(\h)} = \overline{\k^0(\h_{ss})} \cap \overline{\k^0(\h_{nil})}$.  Since elements of $\h_{ss}$ are semisimple, $\overline{\k^0(\h_{ss})} = \z_\k(\h_{ss})$.  Clearly $\overline{\k^0(\h_{nil})} = \k$. Hence $\overline{\k^0(\h)} = \z_\k(\h_{ss})$.
\end{proof}

\begin{proof}[Proof of Theorem \ref{main}] 
Fix an exhaustion $\k = \bigcup_{i \in \Z_{>0}} \k_i$, where each $\k_i$ is a finite-dimensional splittable subalgebra of $\k$.

To show that (\ref{D}) implies (\ref{E}), we must first show that $\h = \z_\k(\h_{ss})$ implies $\h$ is locally nilpotent.  Notice that the equality $\h = \z_\g(\h_{ss})$ implies that every element of $\h_{ss}$ commutes with every element of $\h$.  Now consider a general element $h = h_{ss} + h_{nil} \in \h$.  Choose $k$ such that $(\ad h_{nil})^k = 0$.  For any $x \in \h$, $$(\ad h)^k (x) = (\ad(h_{ss} + h_{nil}))^k(x) = (\ad h_{nil})^k(x) = 0.$$  Hence $\h$ is locally nilpotent.

By Proposition \ref{centralizer} (\ref{toral}), we know $\h_{ss}$ is a toral subalgebra of $\k$.  The equality $\h = \z_{\k}(\h_{ss})$ shows that any semisimple element of $\k$ which centralizes $\h_{ss}$ is already in $\h_{ss}$.  Thus $\h_{ss}$ is a maximal toral subalgebra of $\k$ and (\ref{E}) holds.

To show that (\ref{E}) implies (\ref{D}), we first prove that (\ref{E}) implies that $\h$ is splittable.  Suppose that $\h$ satisfies (\ref{E}).  For any $i \in \Z_{>0}$ note that $$\h \cap \k_i =  \z_\g(\t) \cap \k_i =  \bigcap_{k \geq i}(\z_{\k_k}(\t \cap \k_k) \cap \k_i).$$  Since $\dim \k_i < \infty$, we have
$\h \cap \k_i = \z_{\k_j}(\t \cap \k_j) \cap \k_i$
for some sufficiently large $j \geq i$.
Since $\t \cap \k_j$ is a subalgebra of $\k_j$, we know from \cite[Ch. VII, \S 5, Prop. 3 Cor. 1]{Bourbaki} that $\z_{\k_j}(\t \cap \k_j)$ is a splittable subalgebra of $\k_j$.   
Recall that we have taken $\k_j$ to be splittable also.
Then the intersection $\z_{\k_j}(\h_{ss} \cap \k_j) \cap \k_i$ is splittable, too.  Being a union of splittable algebras, $\h$ is splittable.  

It follows that $\h_{ss} \subseteq \h$.  Then clearly $\t \subseteq \h_{ss}$.  If $\t \neq \h_{ss}$, the existence of a semisimple element $h \in \h \setminus \t$ contradicts the maximality of $\t$.  Therefore $\t = \h_{ss}$, and (\ref{E}) implies (\ref{D}).

Note that (\ref{D}) implies (\ref{F}).  Indeed, suppose $\h = \z_\k(\h_{ss})$.  We have already seen that $\h$ is splittable and locally nilpotent, so by Lemma \ref{lemma1}, $\h = \z_\k(\h_{ss}) = \overline{\k^0(\h)}$.

To show that (\ref{F}) implies (\ref{D}), assume that $\h = \overline{\k^0(\h)}$.  Then clearly $\h$ is locally nilpotent, and we claim that $\h$ is splittable, too.  Indeed, for any $i \in \Z_{>0}$, 
$$\overline{\k^0(\h)} \cap \k_i = \bigcap_{k \geq i} \left(\k_k^0(\h \cap \k_k) \cap \k_i \right).$$  The finite dimensionality of $\k_i$ yields $\overline{\k^0(\h)} \cap \k_i = \k_{j}^0(\h \cap \k_{j}) \cap \k_i$ for some sufficiently large $j \geq i$.  It is well known that $\k_j^0(\h \cap \k_j)$ is a splittable subalgebra of $\k_j$, see \cite[Ch. VII, \S 1, Prop. 11]{Bourbaki}.  Since $\k_i$ is also splittable, the intersection $\k_{j}^0(\h \cap \k_{j}) \cap \k_i$ is splittable, too.   Hence $\overline{\k^0(\h)} \cap \k_i$ is splittable.  Being a union of splittable algebras, $\h$ is splittable.  Therefore Lemma \ref{lemma1} implies $\h = \overline{\k^0(\h)} = \z_\k(\h_{ss})$.

In addition, by Proposition \ref{centralizer} (\ref{self-normalizing}), a subalgebra $\h$ satisfying (\ref{D}) is self-normalizing.  As we have already seen that such a subalgebra is locally nilpotent and splittable, the proof of Theorem \ref{main} is complete.
\end{proof}

We define a subalgebra $\h$ of a splittable subalgebra $\k$ of a locally reductive Lie algebra $\g$ to be a \emph{Cartan subalgebra} if it satisifies any of the equivalent conditions in Theorem~\ref{main}.  Note that since condition (\ref{F}) of Theorem~\ref{main} is intrinsic to $\k$, the definition of a Cartan subalgebra depends only on the isomorphism class of $\k$ and not on the choice of injective homomorphism of $\k$ into a locally reductive Lie algebra.

\centerline{\begin{tabular}{ll}
E.\ D.-C.: & I.\ P.: \\
Department of Mathematics & School of Engineering and Science \\
Rice University & Jacobs University Bremen   \\
6100 S. Main St. & Campus Ring 1 \\
Houston TX 77005-1892 & 28759 Bremen, Germany  \\
{\tt edc@rice.edu} & {\tt i.penkov@jacobs-university.de}
\end{tabular}}


\begin{thebibliography}{}
\bibitem[\textbf{Ba}]{Ba} A.\ Baranov, Finitary simple Lie algebras, J.\ Algebra \textbf{219} (1999), 299--329.
\bibitem[\textbf{BaS}]{BS} A.\ Baranov, H.\ Strade, Finitary Lie algebras, J.\ Algebra \textbf{254} (2002),  173--211.
\bibitem[\textbf{Bo}]{Bourbaki} N.\ Bourbaki, Groupes et Alg\`ebres de Lie, Hermann, Paris, 1975.
\bibitem[\textbf{D}]{D} E.\ Dan-Cohen, Borel subalgebras of root-reductive Lie algebras, J.\ Lie Theory \textbf{18} (2008), 215--241.
\bibitem[\textbf{DPS}]{DPS} E.\ Dan-Cohen, I.\ Penkov, N.\ Snyder, Cartan subalgebras of root-reductive Lie algebras, J.\ Algebra \textbf{308} (2007), 583--611.
\bibitem[\textbf{DPW}]{DPW} E.\ Dan-Cohen, I.\ Penkov, J.\ A.\ Wolf, Parabolic subgroups of real direct limit Lie groups, preprint, arXiv:0901.0295.
\bibitem[\textbf{DP1}]{DP1} I.\ Dimitrov, I.\ Penkov, Weight modules of direct limit Lie algebras, Intern.\ Math.\ Res.\ Notices (1999) No.\ 5, 223--249.
\bibitem[\textbf{DP2}]{DP2} I.\ Dimitrov, I.\ Penkov, Borel subalgebras of $\gl(\infty)$, Resenhas IME-USP \textbf{6} (2004), 153--163.
\bibitem[\textbf{DP3}]{DP4} I.\ Dimitrov, I.\ Penkov, Locally semisimple and maximal subalgebras of the finitary Lie algebras $\gl(\infty)$, $\sl(\infty)$, $\so(\infty)$, and $\sp(\infty)$, preprint, arXiv:0809.2536. 
\bibitem[\textbf{K}]{Karpelevic} F.\ Karpelevi\v c, On nonsemisimple maximal subalgebras of semisimple Lie algebras, (Russian)  Doklady Akad.\ Nauk SSSR (N.S.) \textbf{76} (1951), 775--778.
\bibitem[\textbf{Ma}]{Mackey} G.\ Mackey, On infinite dimensional linear spaces, Trans.\ AMS \textbf{57} (1945) No.\ 2, 155--207.
\bibitem[\textbf{Mo}]{Mostow} G.\ D.\ Mostow, Fully reducible subgroups of algebraic groups, Amer.\ J.\ Math.\ \textbf{78} (1956), 200--221.
\bibitem[\textbf{NP}]{N-P} K.-H.\ Neeb, I.\ Penkov, Cartan subalgebras of $\gl_\infty$, Canad.\ Math.\ Bull.\ \textbf{46} (2003), 597--616.
\end{thebibliography}
\end{document}